\documentclass{amsproc}
\usepackage{geometry} 
\usepackage{amssymb}
\usepackage{faktor}
\usepackage[all,cmtip]{xy}
\usepackage{amsmath,amscd}
\usepackage{graphicx}
\usepackage{mathrsfs}
\usepackage{bbold}
\usepackage{color}
\usepackage{bbm}
\usepackage{enumitem}
\usepackage[mathcal]{euscript}
\usepackage{hyperref}

\newcommand{\Map}{\mathrm{Map}}
\newcommand{\Hom}{\mathrm{Hom}}
\newcommand{\Aut}{\mathrm{Aut}}

\newcommand{\Conf}{\mathrm{Conf}}

\newcommand{\Sym}{\mathrm{Sym}}

\newcommand{\Fun}{\mathrm{Fun}}

\DeclareMathOperator*{\colim}{\mathrm{colim}}

\newcommand \gr{\mathrm{gr}}

\newcommand\id{\mathrm{id}}

\newcommand\C{\mathcal{ C}}
\newcommand\D{\mathcal D}

\newcommand{\Tr}{\mathrm{Pr}}

\newcommand{\Set}{\mathrm{Set}}
\newcommand{\FI}{\mathcal{FI}}
\newcommand{\VI}{\mathcal{VI}}
\newcommand{\FB}{\mathcal{FB}}
\newcommand{\VB}{\mathcal{VB}}
\newcommand{\CSI}{\mathcal{CSI}}
\newcommand{\CSB}{\mathcal{CSB}}

\newcommand{\SCB}{\mathcal{SCB}}
\newcommand{\SCI}{\mathcal{SCI}}
\newcommand{\GB}{\mathcal{GB}}
\newcommand{\GI}{\mathcal{GI}}

\newcommand{\Top}{\mathcal{T}\mathrm{op}}
\newcommand{\Ch}{\mathcal{C}\mathrm{h}}

\newcommand{\Lie}{\mathrm{Lie}}

\newcommand{\Ind}{\mathrm{Ind}}

\newcommand{\Surj}{\mathrm{Surj}}

\newcommand{\Day}{\mathrm{Day}}

\DeclareMathOperator*{\coker}{\mathrm{coker}}

\newcommand\CE{\mathrm{CE}}
\newcommand\Q{\mathrm Q}

\def\treeof(#1;#2){[#1;#2]}

\newcommand{\comp}{\relax}
\def\comp(#1;#2){#1\circ(#2)}

\theoremstyle{remark}
\theoremstyle{definition}\newtheorem{definition}{Definition}[section]
\theoremstyle{theorem}\newtheorem{lemma}[definition]{Lemma}
\theoremstyle{theorem}

\theoremstyle{remark}\newtheorem{remark}[definition]{Remark}
\theoremstyle{definition}
\theoremstyle{definition}
\theoremstyle{definition}
\theoremstyle{definition}
\theoremstyle{definition}\newtheorem{example}[definition]{Example}
\theoremstyle{theorem}\newtheorem{proposition}[definition]{Proposition}
\theoremstyle{theorem}\newtheorem{corollary}[definition]{Corollary}
\theoremstyle{theorem}\newtheorem{theorem}[definition]{Theorem}
\theoremstyle{definition}
\theoremstyle{theorem}

\title{Projection spaces and twisted Lie algebras}
\author{Ben Knudsen}

\begin{document}

\maketitle

\begin{abstract}
A projection space is a collection of spaces interrelated by the combinatorics of projection onto tensor factors in a symmetric monoidal background category. Examples include classical configuration spaces, orbit configuration spaces, the graphical configuration spaces of Eastwood--Huggett, the simplicial configuration spaces of Cooper--de Silva--Sazdanovic, the generalized configuration spaces of Petersen, and Stiefel manifolds. We show that, under natural assumptions on the background category, the homology of a projection space is calculated by the Chevalley--Eilenberg complex of a certain generalized Lie algebra. We identify conditions on this Lie algebra implying representation stability in the classical setting of finite sets and injections.
\end{abstract}

\section{Introduction}

Experience has proven the value of studying configuration spaces in families. When organized correctly, these spaces exhibit emergent algebraic structure that imposes strong constraints on their topological invariants. One important means of organization involves the background space in an essential way; one speaks of operads and their algebras and modules \cite{Idrissi:LSMCS,May:GILS}, or of adding a particle near the boundary of a manifold \cite{McDuff:CSPNP} or onto an edge of a graph \cite{AnDrummond-ColeKnudsen:ESHGBG}. Here, we pursue an orthogonal organizing principle, namely that of the underlying combinatorics.

\subsection{Context and motivation} As a motivating example, consider the ordinary configuration space $\Conf_I(X)$, defined as the space of injections from the finite set $I$ into the topological space $X$. The composite of injections being an injection, the collection of all such spaces forms a presheaf on the category $\FI$ of finite sets and injections. In many examples of interest, this combinatorial structure forces the rational (co)homology of configuration spaces to exhibit \emph{representation stability} \cite{ChurchEllenbergFarb:FIMSRSG}.

There are now a number of machines devoted to the study of stability phenomena in various contexts \cite{Hepworth:OESR,Randal-WilliamsWahl:HSAG}. Unfortunately, most of these machines are adapted to the study of automorphism groups, rather than configuration spaces. The goal of this paper is to develop a framework better adapted to examples such as those listed in the following table.\\

{\begin{center}
\begin{tabular}{| c | c | l | l | l | l | l |}
\hline
 \bf Combinatorics &  \bf Configuration space\\
\hline
Sets& Ordinary\\
\hline
$G$-sets&Orbit\\
\hline
Graphs &Graphical \cite{EastwoodHuggett:ECCP} \\
\hline
Simplicial complexes & Simplicial \cite{CooperdeSilvaSazdanovic:OCSSC}\\
\hline
Collision structures & Generalized \cite{Petersen:CGCS}\\
\hline
Vector spaces & Stiefel\\
\hline
\end{tabular}
\end{center}
}

\vspace{.15in}

\noindent We proceed from the observation that the projection $\Conf_J(X)\to \Conf_I(X)$ induced by the injection $f:I\to J$ is subsumed by the map that splits an $I$-indexed configuration into a configuration indexed by the image of $f$ and one indexed by its complement. The collection of all such splitting maps can be regarded as a kind of cocommutative comultiplication, which, according to the philosophy of Koszul duality, is governed by its Lie algebra of (derived) primitives. 

The resulting interpretation of the homology of ordinary configuration spaces as Lie algebra homology has proven quite fruitful \cite{BrantnerHahnKnudsen:LTTCSI,DrummondColeKnudsen:BNCSS,GadishHainaut:CSWSHPH,Knudsen:BNSCSVFH,Knudsen:HEA,KnudsenMillerTosteson:ESCS}. This success motivates us to generalize the relationship, showing that the homology of something sufficiently like a configuration space is governed by something like a Lie algebra.

\subsection{Framework and results} Our first step is to generalize the relationship between bijections and injections among finite sets in the relationship between a symmetric monoidal category $\C$ and its projection category $\Tr(\C)$. Roughly, a morphism in $\Pr(\C)$ is (opposite to) a projection onto a tensor factor in $\C$ (see Section \ref{section:projection definitions} for details). We then make the following definition, which encompasses all of the spaces listed in the table above (see Section \ref{section:projection examples}).

\begin{definition}\label{def:projection space}
A \emph{projection space} over $\C$ is a functor $X:\Tr(\C)^\mathrm{op}\to\Top$. We say that $X$ is \emph{reduced} if its value on the monoidal unit is a singleton.
\end{definition}

The bulk of the paper is spent in constructing the following composite functor, which associates to each projection space $X$ a $\C$-\emph{twisted Lie algebra} $L(X)$ of \emph{rational primitives}:
\[
\xymatrix{
L:\Fun(\Tr(\C)^\mathrm{op},\Top)\ar[r]^-\cong&\Fun^{\mathrm{oplax}}(\C^\mathrm{op},\Top)\ar[d]^-{A_*}\\
&\Fun^{\mathrm{oplax}}(\C^\mathrm{op},\Ch_\mathbb{Q})\ar[r]^-{\cong}&\mathrm{Coalg}_\mathrm{Com}(\Fun(\C^\mathrm{op},\Ch_\mathbb{Q}))\ar[d]^-\Q\\
&&\mathrm{Alg}_\mathrm{Lie}(\Fun(\C^\mathrm{op},\Ch_\mathbb{Q})).
}
\] Briefly, the first functor witnesses a universal property of the projection category (Corollary \ref{cor:translation equivalence}); the second is the functor of Sullivan chains, a suitable rational replacement for singular chains (Corollary \ref{cor:sullivan chains}); the third relies on a variant of Day convolution (Corollary \ref{cor:right convolution}); and the fourth is Quillen's Koszul duality functor \cite{Quillen:RHT}. The main result is as follows---see Section \ref{section:koszul duality} for undefined terms.

\begin{theorem}\label{thm:main}
Let $X$ be a reduced projection space over the combinatorial symmetric monoidal category $\C$. There is a canonical isomorphism of $\C$-twisted cocommutative coalgebras \[H_*(X;\mathbb{Q})\cong H_*^\mathrm{Lie}(L(X)).\]
\end{theorem}

As a consequence, the homology of generalized configuration spaces is calculated by the Chevalley--Eilenberg complex of a certain Lie algebra. This general connection to Lie algebras may at first seem surprising. Indeed, for ordinary configuration spaces, it is natural to view the connection as arising by combining the Goresky--MacPherson formula, expressing the homology of a stratified space in terms of poset homology, with the identification of the Lie operad with the homology of the partition posets \cite{Petersen:CGCS}. In the more general setting, the resulting poset homology has no obvious relation to the Lie operad, yet the connection persists.

As foreshadowed above, we imagine Theorem \ref{thm:main} as a source and organizing principle for stability phenomena. As a proof of concept, we show that our framework encompasses representation stability. Specifically, stability is tied to the eventual high connectivity of the Lie algebra of rational primitives.

\begin{theorem}\label{thm:stability}
Let $X$ be a reduced $\FI^\mathrm{op}$-space taking values in path connected spaces. Suppose that $L(X)$ satisfies the following conditions. 
\begin{enumerate}
\item $H_i(L(X)_k)$ is finite dimensional for every $i\geq0$ and $k\geq0$.
\item $H_0(L(X)_k)=0$ for $k>1$.
\item $H_i(L(X)_k)=0$ for $i$ fixed and $k$ sufficiently large.
\end{enumerate}
Then $H^*(X;\mathbb{Q})$ is representation stable.
\end{theorem}

We deduce this theorem from a more general result encompassing projection spaces with path disconnected values---see Theorem \ref{thm:general stability}.

\subsection{Future directions}\label{section:future directions} Our work leaves a rather large number of natural questions unanswered (but see \cite{Ho:HRSOCSTCFA} and \cite{Ho:HSDGCS} for closely related work). The following three extensions at least are likely within easy reach.\\
\begin{enumerate}
\item \emph{Integer coefficients.} Our program relies crucially on a strictly symmetric replacement for the oplax monoidal functor of singular chains. In characteristic zero, such a replacement is within easy reach (see Appendix \ref{section:sullivan chains}). With a bit less laziness, one should be able to work integrally, as in \cite{Petersen:CGCS}, after minor modifications of the techniques of \cite{RichterSagave:SCMCAS}---see also Remark \ref{remark:strictly commutative}.\\

\item \emph{Classical primitives}. Let $M$ be a (for simplicity) orientable $n$-manifold. In view of a myriad of (semi-)classical results, it seems a virtual certainty that the Lie algebra associated to the ordinary configuration spaces of $M$ is (quasi-isomorphic to) the tensor product of the compactly supported Sullivan cochains of $M$ with a free twisted Lie algebra on one generator in degree $n-1$ and weight $1$ \cite{FelixThomas:RBNCS,Getzler:RMHMOCS,Knudsen:BNSCSVFH,Petersen:CGCS}. It is likely possible to deduce this claim easily from either of the last two references with a little care taken in comparing models.\\

\item \emph{Orbit primitives.} The results of \cite{BibbyGadish:GFANRSPOCS} suggest that the same description is valid verbatim for orbit configuration spaces, save that the phrase ``free twisted Lie algebra'' should be interpreted with respect to the category of equivariant bijections among free $G$-sets.\\
\end{enumerate}

In more exotic settings, it is less clear what one should expect.\\

\begin{enumerate}[resume]
\item \emph{Generalized primitives}. What is the Lie algebra associated to generalized configuration spaces? Petersen's results in \cite{Petersen:CGCS} suggest an expression in terms of the homology of the order complex associated to a collision structure. One imagines that the Chevalley--Eilenberg complex for this Lie algebra comprises the complexes described by Petersen, in the same way that the various Totaro spectral sequences \cite{Totaro:CSAV} assemble into the Chevalley--Eilenberg complex of the Lie algebra described above in the case of ordinary configuration spaces \cite{Idrissi:LSMCS}.\\
\end{enumerate}

One imagines that Theorem \ref{thm:stability} is merely the beginning of a robust connection between stability phenomena and the twisted Chevalley--Eilenberg complex.\\

\begin{enumerate}[resume]
\item \emph{$\FI$-homology}. We show that sufficient connectivity of $L(X)$ implies representation stability. Is there a direct connection between this Lie algebra and $\FI$-(hyper)homology \cite{ChurchEllenberg:HFM}?\\

\item \emph{Higher stability}. Is there an analogue of Theorem \ref{thm:stability} characterizing higher order representation stability in terms of $L(X)$ \cite{MillerWilson:HORSOCSM}?\\

\item \emph{Beyond representation stability}. What finite generation properties does $L(X)$ dictate for projection spaces in more exotic contexts? What homological asymptotics do they imply? Pursuit of this direction almost certainly entails grappling with thorny Noetherianity problems for $\C$-twisted commutative algebras, already notoriously difficult (and mostly unsolved) in the classical case $\C=\FB$ \cite{NagpalSamSnowden:NSDTTSCA}.\\
\end{enumerate}

It is likely that our framework and examples can be expanded considerably.\\

\begin{enumerate}[resume]
\item \emph{Local coefficients}. One imagines a theory in the vein of \cite{Randal-WilliamsWahl:HSAG} treating the homology of projection spaces with compatibly twisted coefficients. Stability phenomena in the twisted (co)homology of configuration spaces is of considerable interest---see \cite{EllenbergTranWesterland:FNFCQSAMCFF}, for example.\\

\item \emph{Linear collision structures}. It is likely that a linear version of the theory developed in Section \ref{section:collision structures} would permit the expression of the homology of subspace arrangements as twisted Lie algebra homology.\\

\item \emph{Rational homotopy theory}. In the language of \cite{AguiarMahajan:MFSHA}, the Day convolution and pointwise tensor products endow the category of presheaves of chain complexes with the structure of a $2$-\emph{monoidal category}, and the singular chains of a projection space carry the structure of a \emph{double coalgebra} in this category, with the second comultiplication induced by the objectwise diagonal. Does this structure provide a faithful algebraic model for the homotopy theory of rationalized projection spaces \cite{Quillen:RHT}?
\end{enumerate}

\subsection{Conventions} Chain complexes are bounded below and homologically graded. The $r$-fold homological suspension is denoted $[r]$. When working with monoidal categories, we suppress the associator whenever possible. We use the symbols $\otimes$ and $\mathbb{1}$ for the tensor product and unit of a generic monoidal category, employing subscripts (rarely) to disambiguate as necessary (e.g., $\otimes_\C$). An additive tensor category is an additive category equipped with a symmetric monoidal structure whose tensor product distributes over finite coproducts. We write $\Gamma^k(V)=(V^{\otimes k})^{\Sigma_k}$, $\Sym^k(V)=V^{\otimes k}_{\Sigma_k}$, $\Gamma(V)=\bigoplus_{k\geq0}\Gamma^k(V)$, and $\Sym(V)=\bigoplus_{k\geq0}\Sym^k(V)$.

\subsection{Acknowledgements} The author thanks Najib Idrissi and Roberto Pagaria for helpful conversations related to this work and the anonymous referee for her feedback. This paper was written for the proceedings of the conference ``Compactifications, Configurations, and Cohomology,'' held at Northeastern University in October of 2021. The author extends his heartfelt gratitude to the organizers, Peter Crooks and Alex Suciu, for the opportunity to speak, to write, and to be reminded that face-to-face interaction in mathematics is a precious, even indispensable, commodity. This work was supported by NSF grant DMS-1906174.

\section{Projection categories}

In this section, we introduce the projection category $\Tr(\C)$ associated to a monoidal category $\C$. We then explore a selection of simple examples.

\subsection{Definitions}\label{section:projection definitions} Roughly speaking, we wish for a presheaf on the projection category $\Tr(\C)$ to carry a structure map for each projection of an object onto a tensor factor. The following definition makes this idea precise.

\begin{definition}
Let $C_1$ and $C_2$ be objects in the monoidal category $\C$.
\begin{enumerate}
\item A \emph{complementary morphism} from $C_1$ to $C_2$ is the data of an object $D$ of $\C$ together with a morphism $f:C_1\otimes D\to C_2$. 
\item The \emph{composite} of the complementary morphisms $f_1:C_1\otimes D_1\to C_2$ and $f_2:C_2\otimes D_2\to C_3$ is the complementary morphism $f_2\circ (f_1\otimes D_2):C_1\otimes D_1\otimes D_2\to C_2$.
\item An \emph{elementary equivalence} from the complementary morphism $f:C_1\otimes D\to C_2$ to the complementary morphism $f':C_1\otimes D'\to C_2$ is a map $g:D\to D'$ fitting into the commutative diagram
\[
\xymatrix{
C_1\otimes D\ar[dr]_-{f}\ar[rr]^-{C_1\otimes g}&&C_1\otimes D'\ar[dl]^-{f'}\\
&C_2
}
\]
\item We say that two complementary morphisms are \emph{equivalent} if they differ by a finite sequence of elementary equivalences.
\end{enumerate}
\end{definition}

We emphasize that the arrow $g$ in the definition of an elementary equivalence is \emph{not} required to be an isomorphism---see Remark \ref{remark:quillen} below for more on this point.

\begin{remark}
The definition of a complementary morphism is biased by mapping order, i.e., whether to consider $C_1\otimes D\to C_2$ or $C_1\to C_2\otimes D$. Up to opposites, the two choices produce essentially the same result. The definition is also biaed by tensor order, i.e., whether to consider $C_1\otimes D$ or $D\otimes C_1$. Following this section, we will assume that $\C$ is symmetric monoidal, so this bias will play no role.
\end{remark}

\begin{lemma}
Composition is associative and unital up to equivalence and well-defined on equivalence classes.
\end{lemma}
\begin{proof}
We give only an outline, leaving the (easy) details to the reader. For associativity, one uses the associator of $\C$ to produce an elementary equivalence between the two composites, and the compositional unit is the unitor $C\cong C\otimes\mathbb{1}$. Well-definition in one composition factor follows from the commuting diagram \[\xymatrix{
C_1\otimes D_1\otimes D_2'\ar[d]^-{C_1\otimes D_1\otimes g}\ar[r]& C_2\otimes D_2\ar[d]^-{C_2\otimes g}\ar[r]&C_3\ar@{=}[d]\\
C_1\otimes D_1\otimes D_2\ar[r]& C_2\otimes D_2'\ar[r]&C_3,
}\] and a similar diagram establishes well-definition in the second factor.
\end{proof}

The lemma allows us our main definition.

\begin{definition}
Let $\C$ be a monoidal category. The \emph{projection category} of $\C$ is the category $\Tr(\C)$ with the same objects as $\C$ and arrows the equivalence classes of complementary morphisms under composition.
\end{definition}

The projection category is functorial for strong monoidal functors. Briefly, given such a functor $F:\C\to \D$, we define $\Tr(F)=F$ on objects, and we declare that $\Tr(F)$ send the equivalence class of the complementary morphism $C_1\otimes D\to C_2$ to the equivalence class of the composite $F(C_1)\otimes F(D)\cong F(C_1\otimes D)\to F(C_2)$. It is a simple matter to check that $\Tr(F)$ is a well-defined functor.

\begin{remark}\label{remark:bicategory}
The functoriality described above is a shadow of a larger structure. Specifically, the projection category $\Tr(\C)$ is the truncation of an obvious bicategory, and an elaboration of the considerations of Section \ref{section:oplax} shows that a lax monoidal functor between monoidal categories induces a pseudofunctor at the level of bicategories. We make no use of this extended functoriality.
\end{remark}

\begin{remark}\label{remark:quillen}
The projection category is closely related to a construction due to Quillen \cite{Grayson:HAKTII} (see also \cite[p. 11]{Randal-WilliamsWahl:HSAG}). More precisely, the projection category is obtained from the bicategory described in Remark \ref{remark:bicategory} by replacing the morphism categories by their connected components, while Quillen's construction (in the case of a self-action) is obtained by replacing them with the connected components of their maximal subgroupoids---see Example \ref{example:walking arrow} for a simple example illustrating the difference. In particular, one should expect the two to coincide only when $\C$ is itself a groupoid. The use of the full morphism category is essential to the universal property of Theorem \ref{thm:translation equivalence}---see Remark \ref{remark:lax} and Example \ref{example:walking iso}.
\end{remark}

\begin{example}\label{example:walking arrow}
Let $\Delta^1$ denote the walking arrow, i.e., the category with objects $0$ and $1$ and a unique non-identity morphism $e:0\to 1$. The usual rules of integer multiplication extend to a unique symmetric monoidal structure on $\Delta^1$. With this monoidal structure, one finds that there are the exactly five complementary morphisms. Since $\Delta^1$ has no non-identity isomorphisms, these five morphisms are distinct in Quillen's category. One checks that the resulting category is the walking retract, i.e., the category generated by the arrows $i:0\to 1$ and $r:1\to 0$ subject only to the relation $r\circ i=\id$, where $i=[0\otimes 0\xrightarrow{e} 1]$ and $r=[1\otimes 0\xrightarrow{\id} 0]$. In contrast, the commutative diagram \[\xymatrix{
1\otimes 0\otimes 0\ar[dr]^-e\ar[rr]^-{1\otimes e}&&1\otimes 1\ar[dl]_-{\id}\\
&1
}\] is a simple equivalence demonstrating the relation $i\circ r=\id$ in $\Tr(\Delta^1)$. Thus, the projection category is instead the walking isomorphism.
\end{example}

\subsection{Groupoid examples} We begin with the primordial motivating example. Although this result is subsumed in Proposition \ref{prop:groupoid}, we include an independent proof, both for tactility and for later use.

\begin{proposition}\label{prop:fi}
There is a canonical isomorphism of categories \[\Tr(\FB)\cong \FI.\]
\end{proposition}
\begin{proof}
The two categories have the same objects. We define a functor from left to right on arrows by sending the equivalence class of the complementary morphism $f:I_1\sqcup J\cong I_2$ to the injection $f|_{I_1}$. One checks immediately that this prescription is invariant under equivalence and respects identities. As for composition, the complementary morphisms $f:I_1\sqcup J_1\cong I_2$ and $g:I_2\sqcup J_2\cong I_3$ compose to give $g\circ(f\sqcup J_2):I_1\sqcup J_1\sqcup J_2\to I_3$, and \[\left(g\circ (f\sqcup J_2)\right)\mid_{I_1}=g\circ (f\sqcup J_2)\mid_{I_1}=g|_{I_2}\circ f|_{I_1},\] as required. To define the functor from right to left, we send the injection $i:I_1\to I_2$ to the equivalence class of the complementary morphism $I_1\sqcup \left(I_2\setminus \mathrm{im}(i)\right)\cong I_2$. The injection $I=I$ is sent to the equivalence class of $I\sqcup \varnothing\cong I$, which is the identity of $I$ in $\Tr(\FB)$, and the composite of $i:I_1\to I_2$ and $j:I_2\to I_3$ is sent to the equivalence class of $I_1\sqcup \left(I_3\setminus \mathrm{im}(j\circ i)\right)\cong I_3$, which is equivalent to the appropriate composite via the bijection \[I_2\setminus \mathrm{im}(i)\sqcup I_3\setminus \mathrm{im}(j)\cong I_3\setminus \mathrm{im}(j\circ i)\] induced by $j$.

We have shown that both assignments are functors, and the composite functors are the respective identities on objects by construction. On arrows, given $f:I_1\sqcup J\cong I_2$, we have \[I_2\setminus \mathrm{im}(f|_{I_1})=f(J),\] and $f^{-1}|_{f(J)}$ provides an equivalence between our original complementary morphism and the complementary morphism $I_1\sqcup f(J)\cong I_2$, which represents its image under the composite functor. We leave it to the reader to check that the other composite is also the identity on arrows.
\end{proof}

We now consider a class of example recurrent throughout the study of stability phenomena. In the next result, the assumption that $\C$ be skeletal is made solely for simplicity of exposition.

\begin{proposition}\label{prop:groupoid}
Let $\C$ be a skeletal symmetric monoidal groupoid. There is a canonical bijection \[\Hom_{\Tr(\C)}(C_1,C_2)\cong\coprod_{C_2=C_1\otimes D}\Aut(C_2)/\Aut(D)\] under which composition is given by the dashed filler in the following commutative diagram: \[\xymatrix{
\Aut(C_2)\times \Aut(C_3)\ar[d]\ar[r]^-{(-\otimes D_2)\times \id}&\Aut(C_3)\times\Aut(C_3)\ar[r]^-{\circ}&\Aut(C_3)\ar[d]\\
\Aut(C_2)/\Aut(D_1)\times\Aut(C_3)/\Aut(D_2)\ar@{-->}[rr]&&\Aut(C_3)/\Aut(D_1\otimes D_2).
}\]
\end{proposition}
\begin{proof}
For each object $D$ such that $C_2=C_1\otimes D$, an automorphism of $C_2$ determines a complementary morphism $C_1\otimes D\cong C_2$, and every complementary morphism is determined in this way. Thus, we have a canonical surjection \[\coprod_{C_2=C_1\otimes D} \Aut(C_2)\to \Hom_{\Tr(\C)}(C_1, C_2).\] By definition, two automorphisms of $C_2$ differ by an automorphism of the form ${C_1}\otimes g$ for $g\in\Aut(D)$ if and only if the corresponding complementary morphisms are equivalent. It follows that the above surjection descends to the indicated disjoint union of orbit sets, and that each of the resulting functions $\Aut(C_2)/\Aut(D)\to \Hom_{\Tr(\C)}(C_1,C_2)$ is injective. Since $\C$ is skeletal, complementary morphisms indexed by different choices of $D$ are never equivalent, so the function as a whole is injective, implying the first claim. The (essentially immediate) verification of the second claim is left to the reader.
\end{proof}

In the following result, $\FI_G$ denotes the category of equivariant injections among free $G$-sets with finitely many orbits, where $G$ is a fixed group, and $\VI_\mathbb{F}$ denotes the category of linear injections among finite dimensional vector spaces over a fixed vector space $\mathbb{F}$ (resp. $\FB_G$, $\VB_\mathbb{F}$, bijections).

\begin{corollary}\label{cor:groupoid examples}
Let $G$ be a group and $\mathbb{F}$ a field. There are the following canonical isomorphisms of categories:
\begin{align*}
\Tr(\FB_G)&\cong \FI_G\\
\Tr(\VB_\mathbb{F})&\cong \VI_\mathbb{F}.
\end{align*}
\end{corollary}

\subsection{The projection category of a deformation} We close this section with an observation on the relationship between the formation of projection categories and subcategories of a certain type.

\begin{definition}
A subcategory $\iota:\C_0\subseteq\C$ is a \emph{deformation} of $\C$ if there is a functor $R:\C\to \C_0$ and a natural transformation $\tau:\iota\circ R\to \id$.
\end{definition}

Note that any subcategory containing a deformation is itself a deformation.

\begin{proposition}\label{prop:subcategory}
Let $\C$ be a monoidal category and $\C_0\subseteq \C$ a full subcategory containing $\mathbb{1}$ and closed under tensor products. If $\C_0$ is a deformation of $\C$, then the induced functor $\Tr(\C_0)\to \Tr(\C)$ is fully faithful with image the subcategory $\Tr(\C)_0\subseteq \Tr(\C)$ with objects the objects of $\C_0$ and morphisms the equivalence classes of complementary morphisms $f: C_1\otimes D\to C_2$ such that $D$ lies in $\C_0$.
\end{proposition}
\begin{proof}
Prerequisitely, we note that $\Tr(\C)_0$ is, in fact, a subcategory by closure under tensor products. This same assumption, together with fullness and the assumption that $\mathbb{1}\in \C_0$, implies that the tensor product of $\C$ restricts to a monoidal structure on $\C_0$ with the same unit and coherence isomorphisms, so $\Tr(\C_0)$ is defined, and the inclusion $\C_0\subseteq \C$ is strong monoidal, so $\Tr(\C_0)\to \Tr(\C)$ is defined. We claim that this functor factors through the inclusion of $\Tr(\C)_0$ as an isomorphism. The existence, surjectivity on objects, and fullness of the factorization being essentially immediate, the main point is to verify faithfulness. Consider the zig-zag of simple equivalences of the form
\[
\xymatrix{
C_1\otimes D\ar[dr]_-f\ar[r]^-{C_1\otimes g}& C_1\otimes D'\ar[d]^-{f'}&C_1\otimes D''\ar[dl]^-{f''}\ar[l]_-{C_1\otimes g'}\\
&C_2,
}
\] with all objects but $D'$ lying in $\C_0$. By assumption, there is a functor $R:\C\to \C_0$ and a natural transformation $\tau:\iota\circ R\to \id$, from which we derive the enlarged diagram 
\[
\xymatrix{
C_1\otimes R(D)\ar[d]_-{C_1\otimes \tau}\ar[rr]^-{C_1\otimes R(g)}&&C_1\otimes R(D')\ar[d]^-{C_1\otimes \tau}&&C_1\otimes R(D'')\ar[d]^-{C_1\otimes\tau}\ar[ll]_-{C_1\otimes R(g')}\\
C_1\otimes D\ar[drr]_-f\ar[rr]^-{C_1\otimes g}&& C_1\otimes D'\ar[d]^-{f'}&&C_1\otimes D''\ar[dll]^-{f''}\ar[ll]_-{C_1\otimes g'}\\
&&C_2.
}
\] This diagram supplies the following chain of simple equivalences: \[f\sim f\circ C_1\otimes\tau\sim f'\circ C_1\otimes \tau \sim f''\circ C_1\otimes\tau \sim f'',\] all of which lie in $\C_0$. In the same way, an arbitrary zig-zag of simple equivalences expressing the equivalence in $\C$ of two complementary morphisms with source and target in $\C_0$ may be replaced by a (perhaps longer) zig-zag of simple equivalences in $\C_0$. In other words, complementary morphisms in $\C_0$ are equivalent if and only if they are equivalent as complementary morphisms in $\C$, which is faithfulness.
\end{proof}

\section{Collision structures}\label{section:collision structures}

In this section, we develop a general combinatorial framework encompassing all of our examples. Inspired by \cite{Petersen:CGCS}, we define a \emph{collision structure} to be a set of partitions of a finite set closed under merging blocks. We define a notion of morphism between collision structures and calculate the projection category associated to the category of collision structures and bijections (Theorem \ref{thm:collision projection}). 

\subsection{Definitions} We begin by recalling a few standard ideas regarding partitions.

\begin{definition}
Let $I$ be a set. A \emph{partition} of $I$ is a set $P$ of nonempty subsets of $I$, called \emph{blocks}, such that every element of $I$ is contained in exactly one block of $P$. We say that $P$ is a \emph{refinement} of $P'$, written $P\leq P'$, if every block of $P'$ is a union of blocks of $P$.
\end{definition}

The data of a partition $P$ of $I$ is equivalent to the data of the equivalence relation $\sim_P$ on $I$ given by declaring that two elements are equivalent provided they lie in the same block of $P$. Note that the empty set admits a unique partition, which is itself empty.

Partitions of $I$ form a poset $\Pi_I$ under refinement, which has the unique minimum $\{\{i\}\}_{i\in I}$ and, if $I$ is nonempty, the unique maximum $\{I\}$.

\begin{definition}
A \emph{collision structure} on the set $I$ is an upward closed subset $S\subseteq \Pi_I$, i.e., such that $P'\in S$ whenever $P\in S$ and $P\leq P'$.
\end{definition}

We think of a collision structure as prescribing which collisions among elements of the underlying set are forbidden. From this point of view, the requirement of upward closure is obvious.

The intersection or union of collision structures on a fixed set is again a collision structure; therefore, it is sensible to speak of the largest and smallest collision structure with a given property, and of the collision structure generated by a set of partitions.

\begin{example}
The \emph{trivial} collision structure $\varnothing\subseteq \Pi_I$ is initial among collision structures on $I$.
\end{example}

Given a partition $P$ of $I$ and an injection $f:I\to J$, we write $f_*P$ for the partition of $J$ given by the images under $f$ of the blocks of $P$, together with the singletons in the complement of the image of $f$. Given a subset $S\subseteq \Pi_I$, we write $f_*S$ for the collision structure generated by the set $\{f_*P: P\in S\}.$

\begin{definition}
Let $S$ and $T$ be collision structures on $I$ and $J$. We say that an injection $f:I\to J$ is a \emph{map of collision structures} if $f_*S\subseteq T$.
\end{definition}

Since we clearly have $g_*f_*S=(g\circ f)_*S$, maps of collision structures on finite sets form a category $\CSI$. We will also be interested in the wide subcategory $\CSB$ of bijective maps of collision structures.

As we will see below in Section \ref{section:projection examples}, a collision structure $S$ has an associated configuration space, in which collisions among particles are forbidden if the resulting partition lies in $S$. These spaces are precisely the generalized configuration spaces of \cite{Petersen:CGCS}. We close this section by observing that the category of collision structures subsumes the combinatorics underlying the graphical configuration spaces of \cite{EastwoodHuggett:ECCP} and the simplicial configuration spaces of \cite{CooperdeSilvaSazdanovic:OCSSC}

\begin{example}\label{example:graph}
A graph $\Gamma$ determines a collision structure $S_\Gamma$ on its set $V$ of vertices by declaring that $P\in \Pi_V$ lies in $S_\Gamma$ if and only if some block of $P$ contains an edge. An injection between vertex sets is a map of collision structures if and only if it is a graph homomorphism.
\end{example}

\begin{example}\label{example:complex}
A simplicial complex $K$ determines a collision structure $S_K$ on its set $V$ of vertices by declaring that $P\in \Pi_V$ lies in $S_K$ if and only if some block of $P$ is not a simplex of $K$ (we use that the set of simplices is closed under the formation of subsets). An injection between vertex sets is a map of collision structures if and only if it preserves non-simplices.\footnote{Maps of this kind are called ``cosimplicial'' in \cite{CooperdeSilvaSazdanovic:OCSSC}. We avoid this terminology for reasons that are likely obvious.}
\end{example}

Note that the collision structure associated to a graph according to Example \ref{example:graph} coincides with the collision structure associated to its independence complex according to Example \ref{example:complex}.

%


From these examples, we obtain the categories $\GI$ and $\SCI$ of injections among graphs and simplicial complexes, respectively, both full subcategories of $\CSI$, and similarly for bijections.

\subsection{Monoidal structure and projection category} The main result of this section identifies the projection category of the category $\CSB$ of collision structures and bijections. In order to formulate such a result, we first require a monoidal structure.

\begin{definition}
Let $S$ and $T$ be collision structures on $I$ and $J$, respectively. The \emph{disjoint union} of $S$ and $T$ is the collision structure $S\sqcup T\subseteq \Pi_{I\sqcup J}$ generated by $(\iota_I)_*S\cup (\iota_J)_*T$, where $\iota_I:I\to I\sqcup J$ denotes the inclusion (resp. $J$).
\end{definition}

This definition is arranged so that $\iota_I$ and $\iota_J$ are maps of collision structures.

\begin{lemma}
Disjoint union of collision structures extends to a unique symmetric monoidal structure such that the forgetful functor $\CSB\to \FB$ is strong monoidal.
\end{lemma}
\begin{proof}
The main point is to verify that the (suppressed) associator of finite sets is an isomorphism of collision structures, but it is not difficult to check, given collision structures $S_i$ on $I_i$ for $i\in\{1,2,3\}$, that $(S_1\sqcup S_2)\sqcup S_3$ and $S_1\sqcup(S_2\sqcup S_3)$ are both generated by the set $\bigcup_{i=1}^3(\iota_{I_i})_*S_i$.
\end{proof}

By functoriality for strong monoidal functors, we obtain an induced functor at the level of\ projection categories, concerning which we have the following result.

\begin{lemma}\label{lem:faithful}
The induced functor $\Tr(\CSB)\to \Tr(\FB)$ is faithful.
\end{lemma}
\begin{proof}
For $i\in\{1,2\}$, let $S_i$ and $T_i$ be collision structures on $I_i$ and $J_i$, respectively. Consider the following diagram of bijections: \[\xymatrix{
I_1\sqcup J_1\ar[dr]_-{f_1}\ar[rr]^-{I_1\sqcup g}&& I_1\sqcup J_2\ar[dl]^-{f_2}\\
&I_2.
}\] It suffices to show, assuming that $f_1$ and $f_2$ are maps of collision structures (and $g$ not necessarily so), that $f_1$ and $f_2$ are equivalent as complementary morphisms in $\CSB$. Let $T_3$ denote the collision structure on $J_2$ generated by $g_*T_1$ and $T_2$. Then $g$ is a map of collision structures from $T_1$ to $T_3$, and the identity is a map of collision structures from $T_2$ to $T_3$. We claim that $f_2$ is a map of collision structures from $S_1\sqcup T_3$ to $S_2$, for which it suffices to verify the following three containments: \begin{align*}
(f_2)_*(\iota_{I_1})_*S_1&\subseteq S_2\\
(f_2)_*(\iota_{J_2})_*g_*T_1&\subseteq S_2\\
(f_2)_*(\iota_{J_2})_*T_2&\subseteq S_2.
\end{align*} The first and third containment are the assumption that $f_2$ is a map of collision structures, and the second follows from the assumption that $f_1$ is so, since $f_2\circ \iota_{J_2}\circ g=f_1\circ\iota_{J_1}$ by commutativity. We have established the existence of the commutative diagram
\[\xymatrix{
S_1\sqcup T_1\ar[dr]_-{f_1}\ar[r]^-{I_1\sqcup g}& S_1\sqcup T_3\ar[d]^-{f_2}& S_1\sqcup T_2\ar[dl]^-{f_2}\ar[l]_-{\id}\\
&I_2
}\] of collision structures, which represents a pair of elementary equivalences connecting $f_1$ and $f_2$.
\end{proof}

\begin{example}
Via Examples \ref{example:graph} and \ref{example:complex}, the disjoint union of graphs and of simplicial complexes corresponds to the disjoint union of collision structures. In other words, the inclusions $\GB\subseteq \CSB$ and $\SCB\subseteq\CSB$ are strong monoidal. 
\end{example}

We come now to the main result of this section.

\begin{theorem}\label{thm:collision projection}
There is a canonical isomorphism of categories \[\Tr(\CSB)\cong \CSI.\]
\end{theorem}
\begin{proof}
The two categories have the same objects, and we wish to extend this equality to an isomorphism of categories. Forgetting collision structures determines the vertical functors in the diagram \[\xymatrix{
\Tr(\CSB)\ar[d]\ar@{-->}[r]&\CSI\ar[d]\\
\Tr(\FB)\ar[r]^-\cong&\FI,
}\] where the bottom functor is the isomorphism of Proposition \ref{prop:fi}. Since the vertical functors are faithful by Lemma \ref{lem:faithful} and by definition, respectively, the dashed functor (extending the identity on objects) is unique if it exists.

By definition, an arrow from $S_1$ to $S_2$ in the source of the putative functor is represented by a bijection $f:I_1\sqcup J\cong I_2$ such that $f_*(S_1\sqcup T)\subseteq S_2$. We have \[(f|_{I_1})_*S_1=(f\circ\iota_{I_1})_*S_1=f_*(\iota_{I_1})_*S_1\subseteq f_*(S_1\sqcup T)\subseteq S_2,\] so the injection $f|_{I_1}$ defines a morphism from $S_1$ to $S_2$ in $\CSI$. Thus, the dashed functor exists. To obtain its inverse, it suffices for the same reasons to note, given an injection $i:I_1\to I_2$ with $i_*S_1\subseteq S_2$, that the bijection $I_1\sqcup (I_2\setminus \mathrm{im}(i))\cong I_2$ lifts uniquely to a complementary morphism in $\CSB$. For existence, we note that the bijection in question is a map of collision structures when $I_2\setminus \mathrm{im}(i)$ carries the empty collision structure. For uniqueness, we note that, for any collision structure $T$ on $I_2\setminus \mathrm{im}(i)$, any permutation of this set defines a morphism of collision structures $\varnothing\to T$; therefore, any lift of $I_1\sqcup (I_2\setminus \mathrm{im}(i))\cong I_2$ to a complementary morphism is equivalent to the lift constructed above.
\end{proof}

\begin{corollary}
There are the following canonical isomorphisms of categories:
\begin{align*}
\Tr(\GB)&\cong \GI\\
\Tr(\SCB)&\cong \SCI.
\end{align*}
\end{corollary}
\begin{proof}
We aim to use Proposition \ref{prop:subcategory}, the main point being to verify that $\GB$ and $\SCB$ are deformations of $\CSB$. We begin by observing that the assignment of a finite set to its trivial collision structure defines a fully faithful functor $\FB\to \CSB$, which is a section of the forgetful functor. Identifying $\FB$ with its essential image under this functor, we observe that $\FB$ is a deformation of $\CSB$, since the trivial collision structure is initial. Moreover, we have the containment $\FB\subseteq \GB\cap \SCB$; indeed, the trivial collision structure on $I$ corresponds to the discrete graph on $I$ and to the simplex spanned by $I$. Since a subcategory containing a deformation is itself a deformation, the proof is complete.

\end{proof}


\section{Monoidal matters}

In this section, we consider interactions among various tensor products. First, we prove the structural result that the opposite of the projection category represents oplax symmetric monoidal functors with Cartesian target (Corollary \ref{cor:translation equivalence}). Second, we show that, under restrictive but applicable hypotheses, a variant of Day convolution identifies such oplax functors with certain coalgebras (Corollary \ref{cor:right convolution}). These two results form the bridge from combinatorial structure to algebraic structure underpinning our proof of Theorem \ref{thm:main}.

\subsection{Projection categories and lax structures}\label{section:oplax} The goal of this section is to establish a universal property of the projection category in the symmetric setting. Before articulating this universal property in Theorem \ref{thm:translation equivalence} below, we pause to establish notation regarding a few standard concepts.

\begin{definition}\label{def:oplax}
Let $\C$ and $\D$ be monoidal categories.
\begin{enumerate}
\item A \emph{lax (monoidal)} structure on a functor $F:\C\to \D$ is a natural transformation $\mu: F\otimes F\to F\circ\otimes$ and a morphism $\eta:\mathbb{1}\to  F(\mathbb{1})$ such that the following diagrams commute for all $C_1,C_2,C_3\in\C$.
\[\xymatrix{
F(C_1\otimes(C_2\otimes C_3))\ar@{=}[r]^-\sim&F((C_1\otimes C_2)\otimes C_3)\\
F(C_1)\otimes F(C_2\otimes C_3)\ar[u]^-\mu&F(C_1\otimes C_2)\otimes F(C_3)\ar[u]_-\mu\\
F(C_1)\otimes(F(C_2)\otimes F(C_3))\ar[u]^-{F(C_1)\otimes\mu}\ar@{=}[r]^-\sim&(F(C_1)\otimes F(C_2))\otimes F(C_3)\ar[u]_-{\mu\otimes F(C_3)}
}\] 
\[\xymatrix{
F(C_1\otimes\mathbb{1})&F(C_1)\otimes F(\mathbb{1})\ar[l]_-{\mu} \\
F(C_1)\ar@{=}[r]^-\sim\ar@{=}[u]^-\wr&F(C_1)\otimes\mathbb{1}\ar[u]_-{F(C_1)\otimes\eta}.
}\] A \emph{lax (monoidal) functor} is a functor equipped with a (typically suppressed) lax structure.
\item Suppose that $\C$ and $\D$ are symmetric monoidal. We say that the lax functor $F:\C\to \D$ is \emph{symmetric} if the following diagram commutes:
\[\xymatrix{
F(C_1\otimes C_2)\ar@{=}[r]^-\sim&F(C_2\otimes C_1)    \\
F(C_1)\otimes F(C_2)\ar[u]^-{\mu}\ar@{=}[r]^-\sim & F(C_2)\otimes F(C_1)\ar[u]_-{\mu}.
}\] 

\item Let $F_1$ and $F_2$ be lax functors. A natural transformation $\tau:F_1\to F_2$ is \emph{monoidal} if the following diagrams commute:
\[\xymatrix{
F_1(C_1\otimes C_2)\ar[r]^-\tau&F_2(C_1\otimes C_2)&  F_1(\mathbb{1})\ar[rr]^-\tau&&F_2(\mathbb{1})  \\
F_1(C_1)\otimes F_1(C_2)\ar[u]^-{\mu}\ar[r]^-{\tau\otimes\tau}&F_2(C_1)\otimes F_2(C_2)\ar[u]_-{\mu}&&\mathbb{1}\ar[ul]^-\eta\ar[ur]_-\eta
}\]
\item A (symmetric) \emph{oplax} structure on $F:\C\to \D$ is a (symmetric) lax structure on $F^\mathrm{op}$.
\end{enumerate}
\end{definition}

It is immediate from the definitions that monoidality and symmetry is closed under composition; therefore, lax symmetric monoidal functors and monoidal natural transformations form a category $\Fun^{\mathrm{lax}}(\C,\D)$. Likewise, we have the category $\Fun^\mathrm{oplax}(\C,\D)=\Fun^\mathrm{lax}(\C^\mathrm{op},\D^\mathrm{op})^\mathrm{op}$ of oplax symmetric monoidal functors. The reader should note that we do not reflect the condition of symmetry in the notation, since we will have no cause to consider nonsymmetric (op)lax functors.

\begin{theorem}\label{thm:translation equivalence}
Let $\mathcal{C}$ be a symmetric monoidal category and $\mathcal{D}$ a category with finite coproducts. There is a canonical isomorphism of categories \[\Fun(\Tr(\C),\D)\cong\Fun^{\mathrm{lax}}(\C,\D).\]
\end{theorem}

In our application, it will be the dual version of this result that will be of most interest.

\begin{corollary}\label{cor:translation equivalence}
Let $\mathcal{C}$ be a symmetric monoidal category and $\mathcal{D}$ a category with finite products. There is a canonical isomorphism of categories \[\Fun(\Tr(\C)^\mathrm{op},\D)\cong\Fun^{\mathrm{oplax}}(\C^\mathrm{op},\D).\]
\end{corollary}
\begin{proof}
By our assumption on $\D$, the opposite category $\D^\mathrm{op}$ has finite coproducts, so Theorem \ref{thm:translation equivalence} supplies the middle of the three isomorphisms of categories: \[\Fun(\Tr(\C)^\mathrm{op},\D)\cong\Fun(\Tr(\C),\D^\mathrm{op})^\mathrm{op}\cong\Fun^\mathrm{lax}(\C,\D^\mathrm{op})^\mathrm{op}\cong\Fun^\mathrm{oplax}(\C^\mathrm{op},\D).\]
\end{proof}

We begin the proof of the theorem by observing that $\C$ is the source of a natural functor to its own projection category.

\begin{lemma}\label{lem:projection inclusion}
The assignments $\iota(C)=C$ and $\iota(C_1\to C_2)=[C_1\otimes\mathbb{1}\cong C_1\to C_2]$ determine a functor $\iota:\C\to\Tr(\C)$.
\end{lemma}
\begin{proof}
It is immediate that $\iota$ preserves identities. For composition, we appeal to the commutative diagram
\[\xymatrix{
C_1\otimes\mathbb{1}\otimes\mathbb{1}\ar@{=}[r]^-\sim\ar@{=}[d]_-\wr&C_1\otimes\mathbb{1}\ar@{=}[d]_-\wr\ar[r]&C_2\otimes\mathbb{1}\ar@{=}[d]_-\wr\ar@{=}[r]^-\sim&C_2\ar[d]\\
C_1\otimes\mathbb{1}\ar@{=}[r]^-\sim&C_1\ar[r]&C_2\ar[r]&C_3
}\] The composite in the bottom row represents $\iota(C_1\to C_2\to C_3)$, while the full clockwise composite represents $\iota(C_2\to C_3)\circ\iota(C_1\to C_2)$. Since the leftmost vertical arrow is the identity of $C_1$ tensored with the unitor, commutativity implies that the two complementary morphisms are equivalent, as desired. 
\end{proof}

The equivalence of Theorem \ref{thm:translation equivalence} will be given by restriction along the functor $\iota$ of Lemma \ref{lem:projection inclusion}. The following result will allow us to make sense of this idea.

\begin{lemma}\label{lem:oplax automatic}
Let $\D$ be category with finite coproducts.
\begin{enumerate}
\item Given a functor $G:\Tr(\C)\to\D$, the restriction $\iota^*G$ carries a canonical lax structure, which is symmetric.
\item Given a natural transformation $\tau:G_1\to G_2$ in $\Fun(\Tr(\C),\D)$, the restriction $\iota^*\tau$ is monoidal.
\end{enumerate}
\end{lemma}
\begin{proof}
Since the unit in $\D$ is initial, we have the unique map $\mathbb{1}\to G(\mathbb{1})$. The identity and the symmetry of $\C$ determine complementary morphisms from $C_1$ and $C_2$ to $C_1\otimes C_2$, respectively. Applying $G$ and invoking the universal property of the coproduct in $\D$, we obtain a family of maps of the form $G(C_1)\sqcup G(C_2)\to G(C_1\otimes C_2)$. The verification that these maps form a natural transformation, although straightforward, reveals a subtle asymmetry, which may at first be surprising; for this reason, we choose to include some details. The verification of the axioms of a symmetric lax structure and a monoidal natural transformation are left to the enthusiastic reader.
 
By the universal property of the coproduct, naturality in the first variable amounts to the commutativity of the diagrams obtained by applying $G$ to the following two diagrams in $\Tr(\C)$: 
\[\xymatrix{
C_1\ar[d]_-{[C_1\otimes C_2=C_1\otimes C_2]}\ar[rrrrr]^-{[C_1\otimes\mathbb{1}\cong C_1\xrightarrow{f} C_1']}&&&&&C_1'\ar[d]^-{[C_1'\otimes C_2=C_1'\otimes C_2]}\\
C_1\otimes C_2\ar[rrrrr]^{[C_1\otimes C_2 \otimes\mathbb{1}\cong C_1\otimes C_2\xrightarrow{f\otimes C_2} C_1'\otimes C_2]}&&&&&C_1' \otimes C_2.
}\]
\[
\xymatrix{
&&&C_2\ar[dlll]_-{[C_2\otimes C_1\cong C_1\otimes C_2]\qquad}\ar[drrr]^-{\qquad[C_2\otimes C_1'\cong C_1'\otimes C_2]}\\
C_1\otimes C_2\ar[rrrrrr]^-{[C_1\otimes C_2\otimes\mathbb{1}\cong C_1\otimes C_2\xrightarrow{f\otimes C_2} C_1'\otimes C_2]}&&&&&&C_1'\otimes C_2
}
\] The composites in the first diagram are represented by the respective rows of the following commutative diagram in $\C$:
\[\xymatrix{
C_1\otimes\mathbb{1}\otimes C_2\ar@{=}[d]_-\wr\ar@{=}[r]^-\sim&C_1\otimes C_2\ar[r]^-{f\otimes C_2}\ar@{=}[d]&C_1'\otimes C_2\ar@{=}[d]\\
C_1\otimes C_2\otimes\mathbb{1}\ar@{=}[r]^-\sim&C_1\otimes C_2\ar[r]^-{f\otimes C_2}&C_1'\otimes C_2.
}\] As indicated, the symmetry of $\C$ supplies an elementary equivalence between the two, establishing the desired commutativity. On other hand, the counterclockwise composite and righthand diagonal arrow in the second diagram are represented by the clockwise composite and bottom arrow, respectively, of the following commutative diagram in $\C$:
\[
\xymatrix{
C_2\otimes C_1\otimes\mathbb{1}\ar@{=}[d]_-\wr\ar@{=}[r]^-\sim&C_1\otimes C_2\otimes\mathbb{1}\ar@{=}[d]^-\wr\\
C_2\otimes C_1\ar[d]_-{C_2\otimes f}\ar@{=}[r]^-\sim&C_1\otimes C_2\ar[d]^-{f\otimes C_2}
\\
C_2\otimes C_1'\ar@{=}[r]^-\sim&C_1'\otimes C_2.
}
\] The lefthand vertical composite gives a simple equivalence between the two, establishing the claim. Naturality in the second variable is similar.
\end{proof}

Note that the arrow providing the final simple equivlanece of this proof is in general \emph{not an isomorphism}---see Remark \ref{remark:lax} below.

\begin{proof}[Proof of Theorem \ref{thm:translation equivalence}]
Invoking Lemma \ref{lem:oplax automatic}, we obtain the (abusively named) functor \[\iota^*:\Fun(\Tr(\C),\D)\to\Fun^{\mathrm{lax}}(\C,\D),\] which we claim to be an isomorphism. To prove this claim, we construct an inverse isomorphism explicitly. Given a symmetric lax functor $F:\C\to \D$, we define the value of $\overline F:\Tr(\C)\to \D$ on the equivalence class of $f:C_1\otimes D\to C_2$ to be the composite 
\[
F(C_1)\to F(C_1)\sqcup F(D)\xrightarrow{\mu} F(C_1\otimes D)\xrightarrow{F(f)}F(C_2).
\]
One checks immediately that this definition is independent of the choice of representative and that $\overline F$, so defined, is a functor. To conclude, it suffices to verify that the objectwise identities $\overline{\iota^*G}(C)=G(C)$ and $\iota^*\overline F(C)=F(C)$ form a natural transformation and a monoidal natural transformation, respectively. 

Unpacking the definitions, one finds that the value of $\overline{\iota^*G}$ on the equivalence class represented by $C_1\otimes D\to C_2$ is the value of $G$ on the equivalence class of the composite of $C_1\otimes D\otimes\mathbb{1}\cong C_1\otimes D\to C_2$, viewed as a complementary morphism with source $C_2\otimes D$, with $C_1\otimes D=C_1\otimes D$, viewed as a complementary morphism with source $C_1$. This composite is simply $C_1\otimes D\otimes\mathbb{1} \cong C_1\otimes D\to C_2$, viewed as a complementary morphism with source $C_1$, which is equivalent to our original complementary morphism via the unitor $D\otimes\mathbb{1}\cong D$. It follows that $\overline{\iota^*G}=G$.

Similarly, one finds that the value of $\iota^*\overline F$ on a morphism $C_1\to C_2$ is the composite in the top row of the diagram 
\[
\xymatrix{
F(C_1)\ar@{=}[dr]_-\sim\ar[r]&F(C_1)\sqcup F(\mathbb{1})\ar[r]^-\mu&F(C_1\otimes\mathbb{1})\ar@{=}[r]^-\sim&F(C_1)\ar[r]& F(C_2)\\
&F(C_1)\sqcup\mathbb{1}\ar[u]_-{F(C_1)\sqcup\eta}\ar@{=}[urr]_-\sim
}
\]
Since the diagram commutes by our assumption on $F$, this composite is simply the value of $F$ on our original morphism. It follows that $\iota^*\overline F=F$ as bare functors, so it remains to verify that the two lax structures coincide. Since the unit in $\D$ is initial, the second diagram in the definition of a monoidal natural transformation commutes automatically. As for the first, it is an easy exercise to check that the lax structure morphism for $\iota^*\overline{F}$ produced by Lemma \ref{lem:oplax automatic}, as a morphism $F(C_1)\sqcup F(C_2)\to F(C_1\otimes C_2)$ in $\D$, has the same components as $\mu$; we remark only that the argument uses the compatibility of $\mu$ with the respective monoidal symmetries. The claim then follows from the universal property of the coproduct in $\D$.
\end{proof}

\begin{remark}\label{remark:lax}
As our proof shows, the (rather curious) situation is that the universal property of Theorem \ref{thm:translation equivalence} fails to hold for Quillen's category not because one cannot define the desired lax monoidal structure maps, but rather because they do not form a natural transformation unless $\C$ is a groupoid. 
The coarser equivalence relation of $\Tr(\C)$ is precisely what is needed to repair this defect.
\end{remark}

\begin{example}\label{example:walking iso}
We return to Example \ref{example:walking arrow}. By definition, the data of a functor from $\Delta^1$ is that of an arrow $f:X\to Y$ in the target, and our previous analysis shows that the data of an extension of this functor to Quillen's construction is that of a retraction $g$ of $f$. Assuming the target category to have finite coproducts, suppose that we are instead given the data of a lax monoidal structure on the functor in question. The composite arrow \[g: Y\to X\sqcup Y\to X\] is again a left inverse to $f$, as is readily seen from the commutative diagram
\[\xymatrix{
X\sqcup Y\ar[d]\ar[r]^-{f\sqcup Y}&Y\sqcup Y\ar[d]&\varnothing\sqcup Y\ar[l]_-{!\sqcup Y}\\
X\ar[r]^-f& Y.\ar[ur]_-\cong
}\] According to Example \ref{example:walking arrow}, the category $\Tr(\Delta^1)$ is the walking isomorphism, so we should expect that $g$ is also a right inverse; indeed, this fact follows easily from the commutative diagram 
\[\xymatrix{
\varnothing\sqcup X\ar[r]^-{!\sqcup X}&X\sqcup X\ar[d]\ar[r]^-{X\sqcup f}&X\sqcup Y\ar[d]\\
&X\ar[ul]^-\cong\ar@{=}[r]&X.
}\]
\end{example}

\subsection{Day convolution} In this section, we consider two natural tensor product operations defined on functors between symmetric monoidal categories, the left and right Day convolution tensor products. With mild assumptions on the two categories, it is well known that left convolution determines a symmetric monoidal structure on the functor category. A game of opposites then gives conditions under which right convolution determines a symmetric monoidal structure. Although these conditions are significantly more restrictive, they are nevertheless applicable in our setting of interest. 

\begin{definition}
Let $F,G:\C\to\D$ be functors. The left (resp. right) \emph{Day convolution tensor product} of $F$ and $G$ is the left (resp. right) Kan extension $F\otimes G$ in the diagram \[\xymatrix{
\C\times\C\ar[d]_-{\otimes_\C}\ar[r]^-{F\times G}&\D\times\D\ar[r]^-{\otimes_\D}&\D\\
\C\ar@{-->}[urr]_-{F\otimes G}
}\]
\end{definition}

It should be emphasized that either convolution may fail to exist for some or all pairs of functors with fixed source and target. When the left convolution does exist, it is given explicitly by the formula \[(F\otimes G)(C)=\colim\left((\otimes_\C\downarrow C)\to \C\times \C\xrightarrow{F\times G}\D\times\D\xrightarrow{\otimes_\D}\D\right),\] and similarly for right convolution. Mere existence, however, does not guarantee that convolution extends to a monoidal structure.

\begin{theorem}\label{thm:day}
If $\D$ admits, and $\otimes_\D$ distributes over, colimits indexed by $(\otimes_\C\downarrow C)$ for every object $C\in\C$, then left Day convolution extends to a canonical symmetric monoidal structure on $\Fun(\C,\D)$. In this case, there is a canonical isomorphism of categories \[\Fun^{\mathrm{lax}}(\C,\D)\cong \mathrm{Alg}_{\mathrm{Com}}(\Fun(\C,\D)).\]
\end{theorem}

Although this result is essentially standard, it is typically stated with the unnecessarily strong assumption that $\otimes_\D$ distributes over \emph{all} colimits. In the left handed setting, this assumption is almost always satisfied; it is in dualizing to the right handed setting in Corollary \ref{cor:right convolution} that we require the weaker assumption. For these reasons, we have opted to include an outline of the proof---the reader may consult \cite{Day:OCCF,Glasman:DCIC,MandellMaySchwedeShipley:MCDS}, for example, for further details.

\begin{proof}[Proof of Theorem \ref{thm:day}]
The existence assumption grants that $F\otimes G$ is defined for every $F$ and $G$. A unit is supplied by the functor $\mathbb{1}_\Day$ given by the left Kan extension of the inclusion of $\mathbb{1}_\D$ along the inclusion of $\mathbb{1}_\C$. Indeed, we claim that the left Kan extensions indicated by the dashed arrows are as claimed in the following diagrams: \[\xymatrix{
\C\ar[r]^-{(\mathbb{1}_\D, F)}\ar[d]_-{(\mathbb{1}_\C, \id)}&\D\times\D\ar[r]^-{\otimes_\D}&\D&&\C\ar@{=}[dd]\ar[rr]^-F&&\D\\
\C\times\C\ar@{-->}[ur]_-{\mathbb{1}_\Day\times F}\ar[d]_-{\otimes_\C}\\
\C\ar@{-->}[uurr]_-{\mathbb{1}_\Day\otimes F}&&&&\C\ar@{-->}[uurr]_-F
}\] For the righthand diagram, there is nothing to show, and the identification of the innermost Kan extension in the lefthand diagram is immediate from the definitions. Invoking our assumption on distributivity of $\otimes_\D$, it follows that the outer left Kan extension is as indicated. We obtain a left unitor by observing that the left unitors of $\C$ and $\D$ give an isomorphism between the two diagrams, and similarly for a right unitor. The same considerations suffice to identify the left Kan extensions in the following diagram:
\[\xymatrix{
\C\times\C\times\C\ar[d]_-{\otimes_\C\times\id}\ar[rr]^-{F_1\times F_2\times F_3}&&\D\times\D\times\D\ar[r]^-{\otimes_\D\times \id}&\D\times\D\ar[r]^-{\otimes_\D}&\D\\
\C\times\C\ar[d]_-{\otimes_\C}\ar@{-->}[urrr]_-{\quad F_1\otimes F_2\times F_3}\\
\C\ar@{-->}[uurrrr]_-{\quad (F_1\otimes F_2)\otimes F_3}
}\] We obtain an associator after observing that the associators of $\C$ and $\D$ furnish an isomorphism between this diagram and the corresponding diagram for $F_1\otimes(F_2\otimes F_3)$. The same approach furnishes a symmetry. The coherence axioms for $\Fun(\C,\D)$ follow directly from the coherence axioms for $\C$ and $\D$, since each structure morphism in the former was built from the corresponding structure morphisms in the latter two.

Toward the final claim, we observe that, by the universal property of the colimit, a monoid structure on $F:\C\to \D$ provides a compatible collection of maps $F(C_1)\otimes F(C_2)\to F(C)$, one for each arrow $C_1\otimes C_2\to C$. By specializing to the case $C_1\otimes C_2=C$, such a collection furnishes a collection of candidate components $F(C_1)\otimes F(C_2)\to F(C_1\otimes C_2)$ of a lax structure map. Conversely, a lax structure furnishes a candidate monoid structure map via the composites \[F(C_1)\otimes F(C_2)\to F(C_1\otimes C_2)\to F(C).\] Similar remarks apply to units, and it remains to verify that the associativity, symmetry, and unitality constraints of the two structures coincide, a task we leave to the reader.
\end{proof}

Before continuing we record a useful consequence of the proof.

\begin{corollary}\label{cor:day tensor power}
If $\D$ admits, and $\otimes_\D$ distributes over, colimits indexed by $(\otimes_\C\downarrow C)$ for every object $C\in\C$, then the left Day convolution tensor product $\bigotimes_{i=1}^n F_i$ is canonically isomorphic to the left Kan extension in the diagram \[\xymatrix{
\C^n\ar[d]_-{\otimes_\C^{(n)}}\ar[rr]^-{(F_i)_{i=1}^n}&& \D^n\ar[r]^-{\otimes_\D^{(n)}}&\D\\
\C\ar@{-->}[urrr]_-{\otimes_{i=1}^nF_i}
}\]
\end{corollary}

The existence of the right convolution monoidal structure will rely on a technical condition on the structure of overcategories in $\C$, which we refer to as \emph{sparsity}. We record several other related technical conditions for later use. In the following definition, we view finite sets as discrete categories.

\begin{definition}\label{def:comma conditions}
Let $C$ be an object in the symmetric monoidal category $\C$.
\begin{enumerate}
\item We say that $C$ is $n$-\emph{separable} if $(\otimes_\C^{(n)}\downarrow C)$ receives a final functor from a finite set.
\item We say that $C$ is \emph{freely} $n$-\emph{separable} if $(\otimes_\C^{(n)}\downarrow C)$ receives a $\Sigma_n$-equivariant final functor from a finite free $\Sigma_n$-set.
\item We say that $\C$ is \emph{sparse} if every object of $\C$ is $2$-separable.
\end{enumerate}
\end{definition}

\begin{corollary}\label{cor:right convolution}
If $\C$ is a sparse symmetric monoidal category and $\D$ an additive tensor category, then right Day convolution extends to a canonical symmetric monoidal structure on $\Fun(\C^\mathrm{op},\D)$. In this case, there is a canonical isomorphism of categories \[\Fun^{\mathrm{oplax}}(\C^\mathrm{op},\D)\cong \mathrm{Coalg}_{\mathrm{Com}}(\Fun(\C^\mathrm{op},\D)).\]
\end{corollary}
\begin{proof}
In light of the isomorphism $\Fun(\C^\mathrm{op},\D)\cong \Fun(\C,\D^\mathrm{op})^\mathrm{op}$, together with the fact that the formation of the opposite category interchanges limits and colimits, it suffices to show that $\C$ and $\D^\mathrm{op}$ satisfy the hypotheses of Theorem \ref{thm:day}. By sparsity, a colimit indexed by $(\otimes_\C\downarrow C)$ is simply a finite coproduct in $\D^\mathrm{op}$. The proof is complete upon observing that $\D^\mathrm{op}$ is also an additive tensor category, since distributivity over finite coproducts is equivalent to distributivity over finite products by additivity.
\end{proof}

\begin{remark}
The right convolution tensor product seems less well known than its sinistral cousin, perhaps due to the apparent restrictiveness of sparsity. We direct the reader to \cite{Knudsen:HEA} for a prior instance of its use; one assumes that there are other antecedents, but none are known to the author.
\end{remark}

\subsection{Examples}\label{section:day examples} In this section, we specialize the general theory of the preceding section to the examples of interest.


Given a collision structure $S$ on $J$ and a function $f:I\to J$, we write $f^*S$ for the largest collision structure on $I$ for which $f$ is a map of collision structures, i.e.,  \[f^*S=\{P\in\Pi_I: f_*P\in S\}.\] It is easy to check that $f^*S$, so defined, is a collision structure; note, however, that it may be empty. In the case of the inclusion of a subset $I_0\subseteq I$, we write $S|_{I_0}$. 

\begin{example}\label{example:full subgraph}
If $S=S_\Gamma$ is the collision structure associated to a graph $\Gamma$ with vertex set $I$ according to Example \ref{example:graph}, then $S_\Gamma|_{I_0}$ is the collision structure associated to the full subgraph spanned by $I_0$. 
\end{example}

\begin{example}\label{example:full subcomplex}
If $S=S_K$ is the collision structure associated to a simplicial complex $K$ with vertex set $I$ according to Example \ref{example:complex}, then $S_K|_{I_0}$ is the collision structure associated to the full subcomplex spanned by $I_0$.
\end{example}

Write $\Surj(I,n)$ for the set of surjections of $I$ onto the set $\{1,\ldots, n\}$; in other words, $\Surj(I,n)$ is the set of \emph{ordered} partitions of $I$ with $n$ blocks.

\begin{lemma}\label{lem:final inclusion}
Let $S$ be a collision structure on $I$. For every $n\geq0$, there is a canonical inclusion \[\Surj(I,n)\subseteq (\sqcup_{\CSB}^{(n)}\downarrow S),\] which is final if $I$ is non-empty.
\end{lemma}
\begin{proof}
We define the inclusion by sending $p\in\Surj(I,n)$ to the map $\bigsqcup_{j=1}^n S|_{p^{-1}(j)}\to S$ of collision structures given by the canonical bijection $\bigsqcup_{j=1}^np^{-1}(j)\cong I$. Given collision structures $S_j$ on the finite sets $I_j$ and a bijective map $f:\bigsqcup_{j=1}^nS_j \to S$ of collision structures, define 
$p_f:I\to \{1,\ldots, n\}$ by defining $p_f(i)$ to be the index $j$ such that $i\in f(I_j)$; note that there is precisely one such index by bijectivity. Then $p_f^{-1}(j)=f(S_j)$, and we have the commutative diagram:
\[\xymatrix{
\displaystyle\bigsqcup_{j=1}^n I_j\ar[dr]_-f\ar@{-->}[rr]^-{\bigsqcup_{j=1}^n f|_{I_j}}&&\displaystyle\bigsqcup_{j=1}^n p_f^{-1}(j)\ar[dl]^-\cong\\
&I
}\] Finality follows upon noting that the components of the dashed arrow are uniquely determined by commutativity, each component is a map of collision structures, and no such dashed arrow exists for any $p\neq p_f$.
\end{proof}

\begin{corollary}\label{cor:collision separability}
If $S$ is a collision structure on $I$, then $S$ is $n$-separable as an object of $\CSB$ for every $n\geq0$, freely if $I$ is non-empty. In particular, $\CSB$ is sparse.
\end{corollary}
\begin{proof}
Since $\Sigma_n$ acts freely on $\Surj(I,n)$, Lemma \ref{lem:final inclusion} directly implies that any collision structure on a non-empty finite set is freely $n$-separable for every $n\geq0$. For the edge case, we note that $(\sqcup_{\CSB}^{(n)}\downarrow \varnothing)$ is a singleton for every $n\geq0$. 
\end{proof}

\begin{corollary}\label{cor:graph complex separability}
A graph, a simplicial complex, or a finite set is $n$-separable as an object of $\GB$, $\SCB$, or $\FB$, respectively, for every $n\geq0$, freely if non-empty. In particular, all three categories are sparse.
\end{corollary}
\begin{proof}
First, the inclusion of $\Surj(I,n)$ of Lemma \ref{lem:final inclusion} factors through $(\sqcup_{\GB}^{(n)}\downarrow \Gamma)$, $(\sqcup_{\SCB}^{(n)}\downarrow K)$, or $(\sqcup_{\FB}^{(n)}\downarrow I)$, respectively. For the first two, this claim follows from Examples \ref{example:full subgraph} and \ref{example:full subcomplex}, and for the third it is obvious. Since these three subcategories are full, each inclusion is also final.
\end{proof}

Combining these corollaries with Corollary \ref{cor:right convolution}, we see that it is sensible to contemplate right Day convolution in this context. Corollary \ref{cor:day tensor power} and Lemma \ref{lem:final inclusion} yield the following formula for this tensor product.

\begin{corollary}\label{cor:collision day formula}
Let $\D$ be an additive tensor category. Given functors $F_j:\CSB^\mathrm{op}\to \D$ for $1\leq j\leq n$ and a collision structure $S$ on $I$, there is a canonical natural isomorphism 
\[
\left(\bigotimes_{j=1}^n F\right)(S)\cong \bigoplus_{I=\bigsqcup_{j=1}^n I_j} \bigotimes _{j=1}^nF(S|_{I_j}).
\]
\end{corollary}

Examining the proof of Corollary \ref{cor:graph complex separability}, we see that the same formula holds for $\GB$ and $\SCB$ (and $\FB$, of course). By Examples \ref{example:full subgraph} and \ref{example:full subcomplex}, this formula may be interpreted geometrically in terms of full subgraphs and full subcomplexes, respectively.

\section{Twisted Koszul duality}\label{section:koszul duality}

The purpose of this section is to explain how the theory of Koszul duality between Lie algebras and cocommutative coalgebras extends to the twisted setting---see Theorem \ref{thm:duality}. Classically, this theory is developed under two key assumptions: first, that the ground ring is a field of characteristic zero; second, that the graded objects in question are sufficiently connected \cite{FelixHalperinThomas:RHT,Quillen:RHT}. These assumptions serve to guarantee the good behavior of divided powers---their homotopy invariance and nilpotence, respectively. In our setting, the same good behavior is guaranteed instead by the presence of a \emph{weighting} with certain natural properties---see Section \ref{section:weights}. Otherwise, the arguments are essentially identical, and we will at times allow ourselves a certain brevity.

We work throughout over a fixed ground field $\mathbb{F}$, a limitation easily overcome by judicious assumptions of flatness.

\subsection{Weights}\label{section:weights} In this section, we introduce a piece of extra structure that will serve as a substitute for connectivity and divisibility. In the next definition, we refer to Definition \ref{def:comma conditions}.

\begin{definition}
Let $\C$ be a monoidal category.
\begin{enumerate}
\item A \emph{weighting} of $\mathcal{C}$ is a lax monoidal functor $w:\C\to \mathbb{Z}_{\geq0}$. The value $w(C)$ is the \emph{weight} of $C$. The pair $(\C,w)$ is a \emph{weighted monoidal category}
\item A \emph{combinatorial monoidal category} is a weighted monoidal category such that $C\in \C$ is freely $n$-separable for every $n\geq0$ whenever $w(C)>0$.
\end{enumerate}
\end{definition}

Concretely, a weighting amounts to the assignment of a non-negative integer $w(C)$ to each $C\in\C$ such that the following inequalities hold for every $C_1,C_2\in\C$:
\begin{enumerate}
\item $w(C_1)\leq w(C_2)$ whenever $\Hom_\C(C_1,C_2)\neq\varnothing$
\item $w(C_1)+w(C_2)\leq w(C_1\otimes C_2)$.
\end{enumerate}

\begin{example}
Recording the cardinality of a finite set determines a weighting of $\FB$, whence of $\CSB$, $\GB$, and $\SCB$. By the results of Section \ref{section:day examples}, these weighted monoidal categories are combinatorial.
\end{example}

\begin{definition}
Let $\C$ be a weighted monoidal category and $\D$ an additive tensor category. We say that a functor $F:\C\to \D$ is \emph{concentrated in weight $\geq k$} if $F(C)=0$ whenever $w(C)<k$. If $F$ is concentrated in weight $\geq 1$, then we say that $F$ is \emph{reduced}.
\end{definition}

This definition is justified by the following simple result.

\begin{lemma}\label{lem:zero tensor}
If $\D$ is an additive tensor category, then $\D$ admits a zero object $0$, and $0\otimes D\cong 0$ for every $D\in\D$.
\end{lemma}
\begin{proof}
Additivity guarantees the existence of the biproduct indexed by the empty set, which is to say a zero object. Second, since the tensor product of $\D$ distributes over finite coproducts, we have \[0\otimes D\cong\colim\left(\varnothing\xrightarrow{!} \D\right)\otimes D\cong\colim\left(\varnothing\times\D\xrightarrow{!\times \id} \D\times\D\xrightarrow{\otimes}\D\right)\cong\colim\left(\varnothing\xrightarrow{!} \D\right)\cong0.\]
\end{proof}

The comparison to connectivity made above is justified by the following result. 

\begin{lemma}\label{lem:day connectivity}
Let $\C$ be a sparse weighted symmetric monoidal category. If $F_i:\C^\mathrm{op}\to \D$ is concentrated in weight $\geq k_i$ for $i\in\{1,2\}$, then $F_1\otimes F_2$ is concentrated in weight $\geq k_1+k_2$.
\end{lemma}
\begin{proof}
Let $C\in\mathcal{C}$ be an object with $w(C)<k_1+k_2$. From our assumption of sparsity and the definition of right Day convolution, there is a collection of arrows $C_1\otimes C_2\to C$ such that
\begin{align*}
(F\otimes G)(C)&\cong\bigoplus_{C_1\otimes C_2\to C}F_1(C_1)\otimes F_2(C_2)=0,
\end{align*} so it suffices to show that each summand vanishes. We have the inequality \[k_1+k_2>w(C)\geq w(C_1\otimes C_2)\geq w(C_1)+w(C_2).\] Since weights are non-negative, it follows that $w(C_1)<k_1$ or $w(C_2)<k_2$, whence $F_1(C_1)=0$ or $F_2(C_2)=0$.\end{proof}

The comparison to divisibility made above is justified by the following result.

\begin{lemma}\label{lem:symmetric powers homotopy invariant}
Let $\C$ be a combinatorial symmetric monoidal category and $\tau:F_1\to F_2$ a natural transformation between reduced functors from $\C^\mathrm{op}$ to the category of chain complexes in an Abelian tensor category. If $\tau$ is a quasi-isomorphism in weight $< n$, then $\Gamma^k(\tau)$ is a quasi-isomorphism in weight $< n+k-1$ for every $k>0$. In particular, if $k>1$, then $\Gamma^k(\tau)$ is a quasi-isomorphism in weight $\leq n$.
\end{lemma}
\begin{proof}
Fix $C\in\C$ with $w(C)< n+k-1$. If $w(C)=0$, then $T^k(F_i)(C)=0$ by Lemma \ref{lem:day connectivity}, and there is nothing to show, so assume otherwise. Then $C$ is freely $n$-separable by our assumption on $\C$, so Corollary \ref{cor:day tensor power} grants the existence of a finite, $\Sigma_k$-free set $\{f_r:C_{1,r}\otimes\cdots \otimes C_{k,r}\to C\}$ of morphisms such that \[T^k(F_i)(C)\cong\bigoplus_{r}\bigotimes_{j=1}^k F_i(C_{j,r})\] compatibly with $T^k(\tau)$. Since the $F_i$ are reduced, we may assume without loss of generality that $w(C_{j,r})>0$ for every $j$ and $r$. It then follows from the inequality \[n+k-1> w(C)\geq \sum_{j=1}^k w(C_{j,r})\] that $w(C_{j,r})<n$ for every $j$ and $r$. By assumption, $\tau_{C_{j,r}}$ is a quasi-isomorphism, so $T^k(F_i)(C)$ is a quasi-isomorphism. By freeness, this quasi-isomorphism descends to a quasi-isomorphism on $\Sigma_k$-invariants, completing the proof.
\end{proof}

\begin{remark}
In a previous version of this paper, Lemma \ref{lem:symmetric powers homotopy invariant} was stated with the weaker bound only. We thank the referee for pointing out that the proof establishes the improved bound.
\end{remark}

\subsection{Twisted (co)algebraic structures} We come now to the main definitions.

\begin{definition}
Let $\mathcal{C}$ be a sparse symmetric monoidal category. A $\C$-\emph{twisted Lie algebra} (over $\mathbb{F}$) is a Lie algebra in $\Fun(\C^\mathrm{op},\Ch_\mathbb{F})$, regarded as symmetric monoidal under right Day convolution.
\end{definition}

Similarly, one has the notion of a $\C$-twisted cocommutative coalgebra, and so on. In the case $\C=\FB$, one recovers the classical notion of a twisted (co)algebraic structure.

\begin{remark}
As the reader may know, the term ``Lie algebra'' has several inequivalent definitions. Ultimately, we will specialize to a setting in which all such definitions coincide, so the reader is welcome to imagine that her favorite definition is also ours.
\end{remark}

Given a $\C$-twisted cocommutative coalgebra $K$, we write $\overline{K}$ for the kernel of the counit $K\to \mathbb{1}$.

\begin{lemma}
Let $L$ be a $\C$-twisted Lie algebra and $K$ a $\C$-twisted cocommutative coalgebra. 
\begin{enumerate}
\item The unique degree $-1$ coderivation $\partial_\CE$ of $\Gamma(L[1])$ extending the composite \[\Gamma(L[1])[-1]\to \Gamma^2(L[1])[-1]\xrightarrow{[-,-][1]}L[1]\] squares to zero.
\item The unique degree $-1$ derivation $\partial_\Q$ of $\Lie(\overline K[-1])$ extending the composite \[\overline K[-1]\xrightarrow{\Delta[-1]} \Lie^2(\overline K[-1])[1]\subseteq \Lie(\overline K[-1])[1]\] squares to zero.
\end{enumerate}
\end{lemma}
\begin{proof}
We give only a brief outline, the proof being entirely parallel to the corresponding portions of \cite[IV.22(b),(e)]{FelixHalperinThomas:RHT}. For the first claim, since the square of a coderivation of odd degree is again a coderivation, and since $\Gamma(L[1])$ is cofree (as a conilpotent cocommutative coalgebra), it suffices to verify that the composite \[\Gamma^3(L[1])[-2]\to \Gamma^2(L[1])[-1]\to L[1]\] vanishes, which is equivalent to the Jacobi identity. The proof of the second claim is essentially the same save that we appeal instead to coassociativity.
\end{proof}

This result permits the following fundamental definition.

\begin{definition}
Let $L$ be a $\C$-twisted Lie algebra and $K$ a $\C$-twisted cocommutative coalgebra.
\begin{enumerate}
\item The \emph{Chevalley--Eilenberg complex} of $L$ is the $\C$-twisted graded cocommutative coalgebra \[\CE(L)=(\Gamma(L[1]),\partial_\CE+\partial_L).\] Its homology is called the \emph{Lie algebra homology} of $L$ and denoted $H_*^\mathrm{Lie}(L)$.
\item The \emph{Quillen complex} of $K$ is the $\C$-twisted graded Lie algebra \[\Q(K)=(\Lie(\overline K[-1]), \partial_\Q+\partial_K).\]
\end{enumerate}
\end{definition}

\begin{remark}
Some readers may be surprised at the appearance of divided powers rather than symmetric powers in the Chevalley--Eilenberg complex. The two coincide in characteristic zero, and more generally in any setting in which the norm map is an isomorphism on tensor powers (as it will be below).
\end{remark}

These two complexes are closely interrelated.

\begin{lemma}
Let $L$ be a $\C$-twisted Lie algebra and $K$ a $\C$-twisted cocommutative coalgebra. 
\begin{enumerate}
\item The map $\eta:K\to \CE(\Q(K))$ of $\C$-twisted coalgebras induced by the inclusion $\iota: \overline K\cong \overline K[-1][1]\subseteq \Q(K)[1]$ is a chain map.
\item The map $\epsilon:\Q(\CE(L))\to L$ of $\C$-twisted Lie algebras induced by the projection $\pi:\CE(L)[-1]\to L[1][-1]\cong L$ is a chain map.
\end{enumerate}
\end{lemma}
\begin{proof}
Again, we give only a brief outline, as the details are parallel to the relevant portions of the proof of \cite[Thm. 22.9]{FelixHalperinThomas:RHT}. Writing $\delta_1$ and $\delta_2$, respectively, for the clockwise and counterclockwise composites in the diagram  \[\xymatrix{
\overline K\ar[d]_-{\partial_K}\ar[r]^-{\eta}&\CE(\Q(K))\ar[d]^-{\partial_\CE+\partial_{\Q(K)}}\\
\overline K\ar[r]^-{\eta}&\CE(\Q(K)),
}\] the first claim is the equality $\delta_1=\delta_2$. Since $\eta$ is a map of coalgebras, the $\delta_i$ are coderivations; therefore, since $\CE(\Q(K))$ is cofree, it suffices to check that the composite \[\overline K\xrightarrow{\delta_i} \CE(\Q(K))\xrightarrow{\pi} \Q(K)[1]\] is independent of $i$, which is essentially a tautology. The proof of the second claim is similar. 
\end{proof}

\subsection{Duality} In order to state the main result, we require one further definition.

\begin{definition}
Let $\C$ be a sparse weighted symmetric monoidal category. We say that a $\C$-twisted Lie algebra is \emph{reduced} if its underlying functor is so. We say that a $\C$-twisted cocommutative coalgebra $K$ is \emph{reduced} if $\overline{K}$ is so.
\end{definition}

The goal of this section is to prove the following analogue of \cite[Thm. 22.9]{FelixHalperinThomas:RHT}.

\begin{theorem}\label{thm:duality}
If $\C$ is combinatorial, then the maps $\eta$ and $\epsilon$ are quasi-isomorphisms on reduced objects.
\end{theorem}

The main ingredients are the following two lemmas.

\begin{lemma}\label{lem:ce conservative}
If $\C$ is combinatorial, then the Chevalley--Eilenberg complex preserves and reflects quasi-isomorphisms between reduced $\C$-twisted Lie algebras.
\end{lemma}
\begin{proof}
Let $\varphi:L_1\to L_2$ be a map of $\C$-twisted Lie algebras. Assuming that $\varphi$ is a quasi-isomorphism, Lemma \ref{lem:symmetric powers homotopy invariant} grants that $\Gamma^k(\varphi[1])$ is also a quasi-isomorphism for every $k\geq0$. Filtering the Chevalley--Eilenberg complex by tensor degree, we conclude that $\varphi$ induces a quasi-isomorphism at the level of associated graded complexes. By induction and the five lemma, it follows that $\CE(\varphi)$ is a quasi-isomorphism after restriction to any filtration stage; therefore, since direct limits preserve quasi-isomorphisms, $\CE(\varphi)$ itself is a quasi-isomorphism.

For reflection, we adapt the argument of \cite[4.1.9]{FrancisGaitsgory:CKD}. Assuming that $\CE(\varphi)$ is a quasi-isomorphism, let $C\in\mathcal{C}$ be of minimal weight such that $\varphi_C$ is not a quasi-isomorphism. Since the $L_i$ are reduced, we have $n:=w(C)>0$. Setting $R_i:=\coker(L_i[1]\to \CE(L_i))$, we have the following commutative diagram of chain complexes with exact rows: \[\xymatrix{0\ar[r]& L_1(C)[1]\ar[d]_-{\varphi_C[1]}\ar[r]&\CE(L_1)(C)\ar[d]^-{\CE(\varphi)_C}\ar[r]&R_1(C)\ar[r]\ar[d]&0\\
0\ar[r]& L_2(C)[1]\ar[r]&\CE(L_2)(C)\ar[r]&R_2(C)\ar[r]&0.}\] It suffices to show that the righthand vertical arrow is a quasi-isomorphism; indeed, since the middle arrow is a quasi-isomorphism by assumption, the five lemma then implies that the lefthand arrow is a quasi-isomorphism, a contradiction. To this end, we filter $R_i$ by tensor degree, as above, and observe that the associated graded pieces are of the form $\Gamma^k(L_i[1])$ for $k>1$. By minimality, $\varphi$ is a quasi-isomorphism in weight $<n$, so Lemma \ref{lem:symmetric powers homotopy invariant} implies that $\Gamma^k(L_i[1])$ is a quasi-isomorphism in weight $\leq n$ for $k>1$. In particular, the righthand vertical arrow above induces a quasi-isomorphism at the level of associated graded complexes, and the same argument as before completes the proof.
\end{proof}

\begin{lemma}\label{lem:free lie homology}
The projection $\rho:\CE(\Lie(V))\to V[1]\oplus\mathbb{1}$ is a quasi-isomorphism.
\end{lemma}
\begin{proof}
Since adjunctions compose, the universal enveloping algebra of $\Lie(V)$ is the tensor algebra $T(V)$; therefore, the homology of the Chevalley--Eilenberg complex is canonically isomorphic to $\mathrm{Tor}_*^{T(V)}(\mathbb{1},\mathbb{1})$, which may be computed from the two-step complex $T(V)\otimes V\to T(V)$. The claim follows immediately.
\end{proof}

With these results in hand, we complete the proof.

\begin{proof}[Proof of Theorem \ref{thm:duality}]
Filtering $\Q(K)$ by bracket length, the associated graded Lie algebra is $\Lie(\overline K[-1])$, and the induced filtration on $\CE(\Q(K))$ has associated graded object $\CE(\Lie(K[-1]))$. The composite \[\overline K\xrightarrow{\gr(\eta)} \CE(\Lie(\overline K[-1]))\xrightarrow{\rho} \overline K\] is the identity, and $\rho$ is a quasi-isomorphism by Lemma \ref{lem:free lie homology}, so $\gr(\eta)$ is also a quasi-isomorphism. It follows as before that $\eta$ is a quasi-isomorphism. It is easy to check that $\CE(\epsilon)\circ\eta=\id$, so it follows that $\CE(\epsilon)$ is a quasi-isomorphism, and Lemma \ref{lem:ce conservative} implies that $\epsilon$ is also a quasi-isomorphism.
\end{proof}

\begin{remark}
In classical operadic Koszul duality \cite{LodayVallette:AO}, the cobar construction (playing the role of the Quillen complex) fails to preserve quasi-isomorphisms in general, so one obtains instead an equivalence between standard homotopy theory on the algebra side and a nonstandard homotopy theory on the coalgebra side. From this point of view, the main point of this section is that $\Q$ \emph{does} preserve quasi-isomorphisms in our setting, which is the essential content of the reflection clause of Lemma \ref{lem:ce conservative}.
\end{remark}

\section{Projection spaces}

We arrive at last at our primary object of study. After cataloguing examples of projection spaces of interest, we define the functor of rational primitives and prove Theorems \ref{thm:main} and \ref{thm:stability} as stated in the introduction.

\subsection{Examples}\label{section:projection examples} We recall that, according to Definition \ref{def:projection space}, a projection space (over $\C$) is a topological presheaf on $\Tr(\C)$. We begin with the most basic example.

\begin{example}\label{example:products}
Viewing a finite set as a discrete topological space defines a functor $\FI\to \Top$. By restriction along this functor, any presheaf of spaces on $\Top$ determines a projection space over $\FB$ (recall Proposition \ref{prop:fi}); in particular, this construction applies to the representable presheaf $\Map(-,X)$ for any topological space $X$. Under the identification $\Map(I,X)\cong X^I$, one checks that the structure map associated to the injection $f:I\to J$ is the projection $\pi_f:X^J\to X^I$, whose $i$th component is the projection onto the factor of $X^J$ indexed by $f(i)$.
\end{example}

Although this example is rather uninteresting, it gives rise to the main motivating example.

\begin{example}
Given a finite set, the configuration space of $I$-indexed points in $X$ is defined as the subspace $\Conf_I(X)\subseteq X^I$ of injective functions $I\to X$. Since the restriction of an injection along an injection is again an injection, the dashed filler exists in the diagram
\[\xymatrix{
\Conf_J(X)\ar@{-->}[d]\ar[r]&X^J\ar[d]^-{\pi_f}\\
\Conf_I(X)\ar[r]&X^I,
}\] so the collection of all such configuration spaces inherits the structure of a projection space over $\FB$.
\end{example}

The next two examples generalize the previous two in the presence of an action by a group $G$.

\begin{example}
Applying the considerations of Example \ref{example:products} to the inclusion $\FI_G\to \Top_G$ and the functor $\Map_G(-,X)$ represented by the $G$-space $X$, we obtain a projection space over $\FB_G$ (recall Corollary \ref{cor:groupoid examples}). Concretely, the value of this projection space on the finite free $G$-set $I$ is $X^{I_0}$, where we have written $I\cong G\times I_0$ non-canonically. The structure maps combine the action of $G$ with the projections of Example \ref{example:products}.
\end{example}

\begin{example}\label{example:orbit}
Given a finite free $G$-set $I$, the orbit configuration space of $I$-indexed points in $X$ is defined as the subspace $\Conf_I^G(X)\subseteq \Map_G(I,X)$ of injective equivariant functions. Equivalently, writing $I\cong G\times I_0$ (non-canonically), the orbit configuration space is the subspace of $X^{I_0}$ consisting of tuples with pairwise disjoint orbits, i.e., it is the pullback \[\xymatrix{
\Conf_I^G(X)\ar[d]\ar[r]&X^{I_0}\ar[d]\\
\Conf_{I_0}(X_G)\ar[r]&(X_G)^{I_0}.
}\] Since equivariant injections are closed under composition, the first description endows the collection of orbit configuration spaces with the structure of a projection space over $\FB_G$.
\end{example}

There is also a linear analogue of these examples. We fix a topological field $\mathbb{F}$, the cases of $\mathbb{R}$ and $\mathbb{C}$ being of greatest interest.

\begin{example}\label{example:matrices}
By restricting along the inclusion of $\VI_\mathbb{F}$ into the category of topological vector spaces, a presheaf of spaces on the latter determines a projection space over $\VB_\mathbb{F}$ (recall Corollary \ref{cor:groupoid examples}); in particular this construction applies to the representable presheaf $\Hom_\mathbb{F}(-,W)$ for any topological vector space $W$. Since the restriction of an injection along an injection is again an injection, the dashed filler exists in the diagram \[\xymatrix{
V_n(W)\ar@{-->}[d]\ar[r]&\Hom_\mathbb{F}(\mathbb{F}^n,W)\ar[d]^-{\varphi^\vee}\\
V_m(W)\ar[r]&\Hom_\mathbb{F}(\mathbb{F}^m,W),
}\] where $V_n(W)$ denotes the space of linear injections of $\mathbb{F}^n$ into $W$, which is simply the usual non-compact Stiefel manifold of $n$-frames in the real or complex case.
\end{example}

\begin{remark}
One imagines that our framework expands easily to encompass algebrogeometric examples such as flag varieties in positive characteristic.
\end{remark}

Given a partition $P$ of $I$, we write $\Delta_P=\{(x_i)_{i\in I}: i_1\sim_P i_2\implies x_{i_1}=x_{i_2} \}$. Given a collision structure $S$ on $I$, we set $\Delta_S=\bigcup_{P\in S}\Delta_P$ and define the generalized configuration space \cite{Petersen:CGCS} associated to $S$ as \[\Conf_S(X)=X^I\setminus \Delta_S.\] Notice that, if $P\leq P'$, then $\Delta_{P'}\subseteq \Delta_P$, so $\Delta_S=\bigcup_{P\in S_0}\Delta_P$ for any generating set $S_0\subseteq S$.

\begin{lemma}\label{lem:generalized projection}
Let $S$ and $T$ be collision structures on $I$ and $J$, respectively, and $f:I\to J$ an injective map of collision structures. The dashed filler exists in the diagram
\[\xymatrix{
\Conf_T(X)\ar@{-->}[d]\ar[r]&X^J\ar[d]^-{\pi_f}\\
\Conf_S(X)\ar[r]&X^I.
}\]
\end{lemma}
\begin{proof}
By our assumption on $f$, we have the containment $\Delta_{f_*S}\subseteq\Delta_{T}$. On the other hand, we calculate that \begin{align*}\pi_f^{-1}(\Delta_{S})&=\bigcup_{P\in S}\pi_f^{-1}(\Delta_{P})\\
&=\bigcup_{P\in S}\{(x_j)_{j\in J}: i_1\sim_{P} i_2\implies x_{f(i_1)}=x_{f(i_2)}\}\\
&=\bigcup_{P\in S} \{(x_j)_{j\in J}: f(i_1)\sim_{P}f(i_2)\implies x_{f(i_1)}=x_{f(i_2)}\}\\
&=\bigcup_{P\in S}\{(x_j)_{j\in J}: j_1\sim_{f_*P} j_2\implies x_{j_1}=x_{j_2}\}\\
&=\bigcup_{P\in S}\Delta_{f_*P}\\
&=\Delta_{f_*S},
\end{align*} where the last equality uses that $\{f_*P: P\in S\}$ is a generating set for $f_*S$ by definition. Thus, we have the containment $\Conf_T(X)\subseteq X^J\setminus \pi_f^{-1}(\Delta_S)=\pi_f^{-1}(\Conf_S(X))$, as desired.
\end{proof}

\begin{example}\label{example:generalized}
Pulling the projection space $I\mapsto X^I$ of Example \ref{example:products} back along the forgetful functor from collision structures to finite sets, we obtain a projection space over $\CSB$ with the same values. Invoking Lemma \ref{lem:generalized projection}, we obtain a projection space over $\CSB$ extending the assignment $S\mapsto \Conf_S(X)$ on objects. 
\end{example}

\begin{remark}
There is an obvious common generalization of Examples \ref{example:orbit} and \ref{example:generalized}, which one would likely dub a generalized orbit configuration space. The ensemble of such also forms a projection space, where one would likely dub the relevant combinatorial object a $G$-collision structure. 
\end{remark}

\begin{example}
Restricting the projection space of Example \ref{example:generalized} to $\GI$, we obtain a projection space over $\GB$ extending the assignment $\Gamma\mapsto \Conf_{S_\Gamma}(X)$ on objects.
\end{example}

\begin{example}
Restricting the projection space of Example \ref{example:generalized} to $\SCI$, we obtain a projection space over $\SCB$ extending the assignment $K\mapsto \Conf_{S_K}(X)$ on objects.
\end{example}


As the reader will easily verify, all of the examples of projection spaces given in this section are reduced.

\subsection{Proofs of the main results} In the following definition, the reader may take the functor $A_*$ to be any symmetric replacement for the oplax monoidal functor rational singular chains. A specific example of such a replacement is given below in Appendix \ref{section:sullivan chains}.

\begin{definition}
Let $\C$ be a sparse symmetric monoidal category and $X$ a projection space over $\C$. The $\C$-twisted Lie algebra of (derived) \emph{rational primitives} of $X$ is the value on $X$ of the composite functor \[
\xymatrix{
L:\Fun(\Tr(\C)^\mathrm{op},\Top)\ar[r]^-{(\ref{cor:translation equivalence})}&\Fun^{\mathrm{oplax}}(\C^\mathrm{op},\Top)\ar[d]^-{A_*}\\
&\Fun^{\mathrm{oplax}}(\C^\mathrm{op},\Ch_\mathbb{Q})\ar[r]^-{(\ref{cor:right convolution})}&\mathrm{Coalg}_\mathrm{Com}(\Fun(\C^\mathrm{op},\Ch_\mathbb{Q}))\ar[d]^-\Q\\
&&\mathrm{Alg}_\mathrm{Lie}(\Fun(\C^\mathrm{op},\Ch_\mathbb{Q})).
}
\]
\end{definition}

\begin{proof}[Proof of Theorem \ref{thm:main}]
Since the isomorphisms of Corollaries \ref{cor:translation equivalence} and \ref{cor:right convolution} do not change the underlying functor, we have
\begin{align*}
H_*^\mathrm{Lie}(L(X))&= H_*(\CE(\Q(A_*(X))))\\
&\cong H_*(A_*(X))\\
&\cong H_*(X;\mathbb{Q})
\end{align*} by Theorem \ref{thm:duality} and Corollary \ref{cor:sullivan chains}.
\end{proof}

We turn now to stability. The prevailing philosophy that has emerged in the wake of the discovery of representation stability is that stability phenomena are the concrete consequences of finite generation results, representation stability itself being equivalent to finite generation over $\FI$ \cite{ChurchEllenbergFarb:FIMSRSG}. Prerequisitely, then, one requires an action.

A feature of the Chevalley--Eilenberg complex is that it is a symmetric monoidal functor, converting products into tensor products. Thus, as explained at greater length in \cite{KnudsenMillerTosteson:ESCS}, a central Lie subalgebra $L_0\subseteq L$ gives rise to an action of $\Gamma(L_0[1])$ on Lie algebra homology, while an Abelian quotient $L\to L_1$ gives rise to an action of $\Sym(L_1[1])$ on Lie algebra cohomology. Through Theorem \ref{thm:main}, these two constructions give rise to a host of actions on the (co)homology of a projection space, each a potential source of stability phenomena. 

\begin{remark}
In the case $\C=\mathbb{Z}_{\geq0}$, as shown in \cite{KnudsenMillerTosteson:ESCS}, these two types of action are united in the action of a single algebra, called a \emph{transit algebra}. It is less clear how to describe their interaction in general, especially when $\C$ is not a groupoid.
\end{remark}

We now show that this potential is often realized in the classical case $\C=\FB$. A functor from this category (or its opposite) is a symmetric sequence, i.e., a list of objects indexed by the non-negative integers (their weights), together with a $\Sigma_k$-action on the $k$th object for each $k\geq0$. We indicate the weight with a subscript, and we write $V(1)$ for the symmetric sequence with $V(1)_1=V$ and $V(1)_k=0$ for $k\neq 1$.

\begin{theorem}\label{thm:general stability}
If $X$ is a reduced $\FI^\mathrm{op}$-space, then $H^*(X)$ is canonically a $\Sym(H^0(X_1)(1))$-module. This module is finitely generated provided L$(X)$ satisfies the following conditions.
\begin{enumerate}
\item $H_i(L(X)_k)$ is finite dimensional for every $i\geq0$ and $k\geq0$.
\item $H_{-1}(L(X)_k)=0$ for $k>1$ and $H_i(L(X))=0$ for $i<-1$.
\item $H_i(L(X)_k)=0$ for $i$ fixed and $k$ sufficiently large.
\end{enumerate}
\end{theorem}

It is well known that representation stability is equivalent to finite generation over the free twisted commutative algebra on a single generator in degree $0$ and weight $1$, so Theorem \ref{thm:stability} is a special case of this result. In general, the finite generation of the theorem implies a kind of generalized representation stability \cite{Ramos:GRSFIM}.

\begin{remark}
We briefly contextualize the assumptions of the theorem. The third assumption is that $L(X)$ is \emph{eventually highly connected}; one should view this assumption, which resembles standard hypotheses in the study of homological stability, as the key assumption. The first assumption is simply that each component space is of finite type, as will be the case in most examples of interest. The second assumption, which may at first appear the strangest, is in fact also quite reasonable; indeed, it is not hard to show that it holds whenever $X$ is path connected in each weight, or more generally when $H_0(X)$ is cogenerated cofreely by $H_0(X_1)$, as is the case for configuration spaces. One could weaken this assumption at the cost of a more involved statement.
\end{remark}

\begin{lemma}\label{lem:abelian quotient}
If $X$ is a reduced $\FI^\mathrm{op}$-space, then $L(X)$ admits the Abelian twisted Lie algebra $H_0(X_1)(1)[-1]$ as a canonical quotient.
\end{lemma}
\begin{proof}
We begin by observing the canonical isomorphism $L(X)_1\cong A_*(X_1)[-1]$ of chain complexes. Extending by zero, we obtain the composite map \[L(X)\to A_*(X_1)(1)[-1]\to H_0(X_1)(1)[-1],\] where the second composition factor is the projection to the quotient. This map is clearly surjective, and it is a map of Lie algebras; indeed, $L(X)$ is reduced, and weight is additive under the bracket.
\end{proof}

\begin{proof}[Proof of Theorem \ref{thm:general stability}]
The Chevalley--Eilenberg complex is the total complex of a bicomplex, whose differentials reflect the internal differential of the Lie algebra and the Lie bracket, respectively. One of the two spectral sequences associated to the dual of this bicomplex has the form
\[E_1\cong \Sym(H^*(L(X))[1])\implies H^*(X;\mathbb{Q})\] (we use our first assumption to guarantee that the dual of the divided power is the symmetric power). This spectral sequence is a spectral sequence of $\Sym(H^0(X_1)(1))$-modules by functoriality, since the action arises from a map of Lie algebras. As such a module, the $E^1$-page is freely generated by $\Sym(V)$, where \[V=\frac{H^*(L(X))[1]}{H^0(L(X_1)(1))[1]}.\] By our first assumption, $H^0(X_1)$ is finite dimensional, so $\Sym(H^0(X_1)(1))$ is Noetherian \cite{Snowden:SSEDM}; thus, it suffices to show that $\Sym(V)$ is finite dimensional in fixed degree $i$, for which we calculate that
\[
\Sym(V)_{k,i}=\bigoplus_{r\geq0}\left(\bigoplus_{k_1+\cdots+k_r=k}\bigoplus_{i_1+\cdots+i_r=i}\Ind_{\prod_{j=1}^r \Sigma_{k_j}}^{\Sigma_k}\bigotimes_{j=1}^r V_{k_j,i_j}\right)_{\Sigma_r}.\] For fixed $k$, this expression is finite dimensional by reducedness and our first assumption, so it suffices to show that it vanishes for $k$ sufficiently large. By our third assumption on $L(X)$, there exists $\ell$ sufficiently large so that $V_{k,\leq i}=0$ for $k\geq \ell$; therefore, the summand indexed by $r$ vanishes for $k\geq r\ell$. On the other hand, our second assumption on $L(X)$ implies that $V_{k,i}=0$ for $i\leq0$, so we may take $i_j>0$ for $1\leq j\leq r$ in the above expression. It follows that the summands indexed by $r>i$ all vanish. Combining these observations, we conclude that $\Sym(V)_{k,i}=0$ for $k\geq i\ell$, as desired.
\end{proof}

\begin{remark}
Let $M$ be a (for simplicity) orientable manifold of dimension $n$. According to the conjecture articulated in Section \ref{section:future directions}(2), the homology of the Lie algebra of rational primitives of the ordinary configuration spaces of $M$ can be described in weight $k$ as the vector space $H^{-*}(M)\otimes \mathrm{Lie}(k)[k(n-1)]$, where $\Lie(k)$ is the $k$th Lie representation. As long as $n>1$, this Lie algebra satisfies the assumptions of Theorem \ref{thm:general stability}, so we recover the primordial example of representation stability. This analysis does not actually depend on the validity of the conjecture, since the Chevalley--Eilenberg complex of this Lie algebra does calculate the correct homology.
\end{remark}

\begin{appendix}

\section{Sullivan chains}\label{section:sullivan chains} The cup product arises from the Alexander--Whitney map $C_*(X\times Y)\to C_*(X)\otimes C_*(Y)$, a natural transformation constituting an oplax monoidal structure that is famously not symmetric. Fortunately, the failure of symmetry is governed by highly coherent homotopies, which are strictifiable in characteristic zero for abstract reasons. 

The purpose of this appendix is to present an explicit model for this strictification. In brief, we construct a pre-dual of (a completion of) the functor of Sullivan cochains $A^*$, a reference for which is \cite[II.10(c)]{FelixHalperinThomas:RHT}, where it is denoted $A_{PL}$. It is common to discuss this complex in barycentric coordinates, but the expression in increasing coordinates given in \cite{RichterSagave:SCMCAS} will be more convenient for our purposes.

Write $V_n$ for the chain complex with basis $\{x_1,\ldots, x_n, dx_1,\ldots, dx_n\}$, where $|x_i|=0$. This complex carries a natural simplicial structure via the maps 
\[
\partial_i(x_k)=\begin{cases}
x_k&\quad k\leq i\\
x_{k-1}&\quad k>i
\end{cases}\qquad\qquad s_j(x_k)=\begin{cases}
x_k&\quad k\leq j\\
x_{k+1}&\quad k>j.
\end{cases}
\] The dual $V_n^\vee$ thus carries a natural cosimplicial structure.

\begin{definition} 
Fix $n\geq0$.
\begin{enumerate}
\item The complex of \emph{Sullivan cochains} on the standard $n$-simplex $\Delta^n$ is the commutative differential graded algebra $A^*(\Delta^n)$ freely generated by $V_n$.
\item The complex of \emph{Sullivan chains} on the standard $n$-simplex $\Delta^n$ is the conilpotent cocommutative differential graded coalgebra cofreely cogenerated by $V_n^\vee$.
\item The complex of \emph{completed Sullivan cochains} on $\Delta^n$ is $\widehat A^*(\Delta^n)=A_*(\Delta^n)^\vee,$ regarded as a commutative differential graded algebra.
\end{enumerate}
\end{definition}

Each of these constructions inherits a (co)simplicial structure; indeed, each is obtained by applying a functor pointwise to a (co)simplicial object.

\begin{definition}
Write $\iota:\Delta\to \Fun(\Delta^\mathrm{op},\Set)$ for the Yoneda embedding.
\begin{enumerate}
\item
The functor of \emph{Sullivan cochains} is the left Kan extension $A^*$ in the diagram \[\xymatrix{\Delta\ar[d]_-\iota\ar[r]^-{A^*(\Delta^\bullet)}&\mathrm{Alg}_{\mathrm{Com}}(\Ch_\mathbb{Q})^\mathrm{op}\\
\Fun(\Delta^\mathrm{op}, \mathrm{Set})\ar@{-->}[ur]_-{\quad A^*=\iota_!A^*(\Delta^\bullet)}
}\] (resp. completed Sullivan cochains, $\widehat A^*$).
\item The functor of \emph{Sullivan chains} is the left Kan extension $A_*$ in the diagram 
\[\xymatrix{\Delta\ar[d]_-\iota\ar[r]^-{A_*(\Delta^\bullet)}&\mathrm{Coalg}_{\mathrm{Com}}(\Ch_\mathbb{Q})\\
\Fun(\Delta^\mathrm{op}, \mathrm{Set}).\ar@{-->}[ur]_-{\quad A_*=\iota_!A_*(\Delta^\bullet)}
}\] 
\end{enumerate}
The complex of Sullivan (co)chains on the topological space $X$ is the value of the corresponding functor on the singular set of $X$.
\end{definition}

Note the canonical isomorphism $\widehat A^*\cong A_*^\vee$; indeed, dualization is colimit preserving when viewed as a functor from cocommutative differential graded coalgebras to the opposite category of commutative differential graded algebras.

\begin{lemma}\label{lem:completed quasi}
The map $A^*\to \widehat A^*$ induced by the composite $V_\bullet\to (V_\bullet^\vee)^\vee\to \widehat A^*(\Delta^\bullet)$ is a quasi-isomorphism.
\end{lemma}
\begin{proof}
By \cite[Prop. 10.5]{FelixHalperinThomas:RHT}, it suffices to check that the induced map $A^*(\Delta^n)\to \widehat A^*(\Delta^n)$ is a quasi-isomorphism for each $n\geq0$. Since the unique map $0\to V_n$ is a quasi-isomorphism, and since symmetric powers preserve quasi-isomorphisms over $\mathbb{Q}$, it follows that the unit of $A^*(\Delta^n)$ is a quasi-isomorphism. Since the linear dual of a quasi-isomorphism is also a quasi-isomorphism, it follows that the unique map $V_n^\vee\to 0$ is a quasi-isomorphism; therefore, after invoking the same property of symmetric powers a second time, it follows that the counit of $A_*(\Delta^n)$ is a quasi-isomorphism. By the same property of the linear dual, it now follows that the unit of $\widehat{A}^*(\Delta^n)$ is a quasi-isomorphism. Since the map in question is compatible with units, being a map of algebras, the claim follows.
\end{proof}

\begin{theorem}\label{thm:sullivan chains}
Let $K$ be a simplicial set and write $N_*$ for the functor of rational normalized chains. There is a natural quasi-isomorphism $A_*(K)\simeq N_*(K)$ of differential graded coalgebras.
\end{theorem}
\begin{proof}
The proof closely follows that of \cite[Thm. 10.9]{FelixHalperinThomas:RHT}. At the level of simplices, the respective counits induce the maps \[A_*(\Delta^\bullet)\leftarrow A_*(\Delta^\bullet)\otimes N_*(\Delta^\bullet)\to  N_*(\Delta^\bullet)\] of cosimplicial differential graded coalgebras. Since both are objectwise acyclic, these maps are quasi-isomorphisms. Kan extending, one obtains the zig-zag \[A_*=\iota_!A_*(\Delta^\bullet)\leftarrow \iota_!\left(A_*(\Delta^\bullet)\otimes N_*(\Delta^\bullet)\right)\to  \iota_!N_*(\Delta^\bullet)\cong N_*,\] where the rightmost isomorphism is justified by the observations that colimits of coalgebras are computed at the level of chain complexes; that $N_*$ is colimit-preserving; and that the identity functor is the left Kan extension of the Yoneda embedding along itself. One could argue directly that these maps are quasi-isomorphisms by appealing to a cosimplicial version of \cite[Prop. 10.5]{FelixHalperinThomas:RHT}. Alternatively, consider the following commutative diagram of functors from simplicial sets to commutative differential graded algebras: \[\xymatrix{
\widehat A^*\ar[r]& \iota_!\left(\left(A_*(\Delta^\bullet)\otimes N_*(\Delta^\bullet)\right)^\vee\right)& N^*\ar[l]\\
&\iota_!(\widehat A^*(\Delta^\bullet)\otimes N^*(\Delta^\bullet))\ar[u]\\
A^*\ar[uu]\ar[r]&\iota_!\left(A^*(\Delta^\bullet)\otimes N^*(\Delta^\bullet)\right)\ar[u]&N^*,\ar[l]\ar@{=}[uu]
}\] the top row is the linear dual of the zig-zag of interest, the maps in the bottom row are shown to be quasi-isomorphisms in \cite[Thm. 10.9]{FelixHalperinThomas:RHT}, and the lefthand vertical arrow is a quasi-isomorphism by Lemma \ref{lem:completed quasi}; therefore, since formation of the linear dual reflects quasi-isomorphisms, it suffices to show that the vertical maps in the middle column are quasi-isomorphisms. By \cite[Prop. 10.5]{FelixHalperinThomas:RHT}, it suffices to show that the maps
\[A^*(\Delta^n)\otimes N^*(\Delta^n)\to \widehat A^*(\Delta^n)\otimes N^*(\Delta^n)\to \left(A_*(\Delta^n)\otimes N_*(\Delta^n)\right)^\vee\] are quasi-isomorphisms for each $n\geq0$. The first is the quasi-isomorphism of Lemma \ref{lem:completed quasi} tensored with the identity. The second is the canonical map from the tensor product of the linear duals to the linear dual of the tensor product, which is a quasi-isomorphism on complexes with finite dimensional homology.
\end{proof}

\begin{corollary}\label{cor:sullivan chains}
Let $X$ be a topological space. There is a natural quasi-isomorphism \[A_*(X)\simeq C_*(X;\mathbb{Q})\] of differential graded coalgebras and of oplax monoidal functors.
\end{corollary}
\begin{proof}
Restricting the quasi-isomorphism of Theorem \ref{thm:sullivan chains} along the singular set functor, we obtain the desired quasi-isomorphism of coalgebras. To complete the proof, we need only recall that the coalgebra structure on singular chains determines the oplax monoidal structure via \[C_*(X\times Y)\to C_*(X\times Y)\otimes C_*(X\times Y)\xrightarrow{C_*(\pi_1)\otimes C_*(\pi_2)} C_*(X)\otimes C_*(Y).\]
\end{proof}

\begin{remark}\label{remark:strictly commutative}
One imagines that our approach in this appendix translates essentially unchanged to the setting of \cite{RichterSagave:SCMCAS}, providing a symmetric model for integral singular chains. We have not checked the details.
\end{remark}

\end{appendix}

\bibliographystyle{plain}
\bibliography{references}

\end{document}